\documentclass[a4paper,11pt]{amsart}
\usepackage[left=2.7cm,right=2.7cm,top=3.5cm,bottom=3cm]{geometry}

\usepackage{amsthm,amssymb,amsmath,amsfonts,mathrsfs,amscd,dsfont,verbatim}
\usepackage[latin1]{inputenc}
\usepackage[all,cmtip]{xy}
\usepackage{latexsym}
\usepackage{longtable}
\usepackage{mathtools}

\mathtoolsset{showonlyrefs}

\usepackage{graphicx}
\newcommand{\Bmu}{\mbox{$\raisebox{-0.59ex}
  {$l$}\hspace{-0.18em}\mu\hspace{-0.88em}\raisebox{-0.98ex}{\scalebox{2}
  {$\color{white}.$}}\hspace{-0.416em}\raisebox{+0.88ex}
  {$\color{white}.$}\hspace{0.46em}$}{}}

\numberwithin{equation}{section}

\newfont{\cyr}{wncyr10 scaled 1100}
\newfont{\cyrr}{wncyr9 scaled 1000}

\theoremstyle{plain}
\newtheorem{theorem}{Theorem}[section]
\newtheorem*{theoremA}{Theorem A}
\newtheorem*{theoremB}{Theorem B}
\newtheorem{proposition}[theorem]{Proposition}
\newtheorem{lemma}[theorem]{Lemma}
\newtheorem{corollary}[theorem]{Corollary}
\newtheorem{conjecture}[theorem]{Conjecture}

\theoremstyle{definition}
\newtheorem{definition}[theorem]{Definition}
\newtheorem{assumption}[theorem]{Assumption}

\theoremstyle{remark}
\newtheorem{remark}[theorem]{Remark}

\newcommand{\Q}{\mathds Q}
\newcommand{\N}{\mathds N}
\newcommand{\Z}{\mathds Z}
\newcommand{\R}{\mathds R}
\newcommand{\C}{\mathds C}

\newcommand{\F}{\mathds F}
\newcommand{\T}{\mathds T}
\newcommand{\A}{\mathds A}

\newcommand{\defeq}{\vcentcolon=}
\newcommand{\eqdef}{=\vcentcolon}

\DeclareMathOperator{\End}{End}

\DeclareMathOperator{\Frob}{Frob}

\DeclareMathOperator{\Hom}{Hom}

\DeclareMathOperator{\Gal}{Gal}
\DeclareMathOperator{\GL}{GL}

\DeclareMathOperator{\SL}{SL}
\DeclareMathOperator{\Sel}{Sel}

\DeclareMathOperator{\Sym}{Sym}

\DeclareMathOperator{\CH}{CH}
\DeclareMathOperator{\AJ}{AJ}

\DeclareMathOperator{\Ind}{Ind}
\DeclareMathOperator{\h}{\boldsymbol h}
\DeclareMathOperator{\cyc}{cyc}
\DeclareMathOperator{\ddiv}{div}

\DeclareMathOperator{\im}{im}

\newcommand{\res}{\mathrm{res}}
\newcommand{\cores}{\mathrm{cores}}

\newcommand{\tr}{\mathrm{tr}}
\newcommand{\ord}{\mathrm{ord}}
\newcommand{\new}{\mathrm{new}}

\newcommand{\an}{\mathrm{an}}
\newcommand{\alg}{\mathrm{alg}}

\newcommand{\Sha}{\mbox{\cyr{X}}}

\usepackage[usenames]{color}
\definecolor{Indigo}{rgb}{0.2,0.1,0.7}
\definecolor{Violet}{rgb}{0.5,0.1,0.7}
\definecolor{White}{rgb}{1,1,1}
\definecolor{Green}{rgb}{0.1,0.9,0.2}

\newcommand{\longmono}{\mbox{\;$\lhook\joinrel\longrightarrow$\;}}

\newcommand{\longepi}{\mbox{\;$\relbar\joinrel\twoheadrightarrow$\;}}

\newcommand{\E}{\mathcal E}

\newfont{\gotip}{eufb10 at 12pt}

\newcommand{\cO}{{\mathcal O}}

\newcommand{\cG}{{\mathcal G}}

\newcommand{\m}{\mathfrak{m}}
\newcommand{\p}{\mathfrak{p}}
\newcommand{\fP}{\mathfrak{P}}

\newcommand{\hf}{\boldsymbol f^{(p)}}

\DeclareMathOperator{\GS}{GS}
\DeclareMathOperator{\Ta}{Ta}
\DeclareMathOperator{\PSL}{PSL}

\DeclareMathOperator{\sspp}{sp}
\DeclareMathOperator{\bsspp}{\overline{sp}}

\DeclareMathOperator{\rank}{rank}
\DeclareMathOperator{\corank}{corank}

\DeclareMathOperator{\f}{\boldsymbol f}

\include{thebibliography}

\begin{document}

\title[On Shafarevich--Tate groups and analytic ranks in Hida families]{On Shafarevich--Tate groups and analytic ranks\\in families of modular forms, I. Hida families}
\author{Stefano Vigni}


\thanks{}

\begin{abstract}
This is the first article in a two-part project whose aim is to study algebraic and analytic ranks in $p$-adic families of modular forms. Let $f$ be a newform of weight $2$, square-free level and trivial character, let $A_f$ be the abelian variety attached to $f$ and for every good ordinary prime $p$ for $f$ let $\hf$ be the $p$-adic Hida family through $f$. We prove that, for all but finitely many primes $p$ as above, if $A_f$ is an elliptic curve such that the Mordell--Weil group $A_f(\Q)$ has rank $1$ and the $p$-primary part of the Shafarevich--Tate group of $A_f$ over $\Q$ is finite, then all specializations of $\hf$ of weight congruent to $2$ modulo $2(p-1)$ and trivial character have finite $p$-primary Shafarevich--Tate group and $1$-dimensional image of the relevant $p$-adic \'etale Abel--Jacobi map. An analogous result is obtained also in the rank $0$ case. As a second contribution, with no restriction on the dimension of $A_f$ but assuming the non-degeneracy of certain height pairings \emph{\`a la} Gillet--Soul\'e between Heegner cycles, we show that, for all but finitely many $p$, if $f$ has analytic rank $1$, then all specializations of $\hf$ of weight congruent to $2$ modulo $2(p-1)$ and trivial character have analytic rank $1$. This result provides some evidence in rank $1$ and weight larger than $2$ for a conjecture of Greenberg predicting that the analytic ranks of even weight modular forms in a Hida family should be as small as allowed by the functional equation, with at most finitely many exceptions.
\end{abstract}


\address{Dipartimento di Matematica, Universit\`a di Genova, Via Dodecaneso 35, 16146 Genova, Italy}
\email{stefano.vigni@unige.it}

\subjclass[2020]{11F11 (primary), 14C25 (secondary)}

\keywords{Modular forms, Hida families, big Heegner points, Heegner cycles, $L$-functions, Shafarevich--Tate groups.}

\maketitle


\section{Introduction} \label{Intro}

The theme of the present article is the study of certain arithmetic invariants of modular forms (algebraic ranks, analytic ranks, Shafarevich--Tate groups) when the modular forms they are attached to vary in a Hida (\emph{i.e.}, slope $0$) family. This is the first paper in a two-part project: the article \cite{PPV} deals with Coleman (\emph{i.e.}, finite slope) families.

In the past few decades, Hida theory has played a prominent role in algebraic number theory and arithmetic geometry. On the one hand, deforming an eigenform $f$ (often, but not necessarily, of weight $2$) to a Hida family turns out to be a powerful technique to get information on $f$. Just as sample achievements for which Hida theory was a key ingredient, here we mention the work of Greenberg--Stevens (\cite{GS}) on the exceptional zero conjecture of Mazur--Tate--Teitelbaum (\cite{MTT}) for weight $2$ newforms and the proof by Bertolini--Darmon of the rationality of Stark--Heegner points over genus fields of real quadratic fields (\cite{BD-Stark-Heegner}). On the other hand, Hida theory provides a means to package modular forms (and their Galois representations) in a systematic, continuous manner, so it is natural to investigate, as in our paper, how arithmetic objects or invariants attached to modular forms vary in a Hida family, with the ultimate goal of deducing results on all (or, at least, a significant collection of) forms in the family from information on a single form in the family. This is the point of view taken, for example, by Emerton--Pollack--Weston in their study of the variation of (cyclotomic) Iwasawa invariants in Hida families (\cite{EPW}; \emph{cf.} also \cite{CKL} for analogous results in an anticyclotomic setting) and by Howard in his article on the variation of Heegner points in Hida families (\cite{Howard-Inv}).

Let $\f$ be a $p$-adic Hida family of tame level $N$, where $p$ is a prime number such that $p\nmid2N$. By definition, $\f\in\mathcal R[\![q]\!]$ where $\mathcal R$ is a complete local noetherian domain that is finite and flat over a suitable $p$-adic Iwasawa algebra. A crucial property of $\f$ is that it admits specializations $f_\wp$ at \emph{arithmetic} prime ideals $\wp$ of $\mathcal R$: every $f_\wp$ is a cuspidal eigenform on $\Gamma_1(Np^r)$ for some $r\geq1$ (see \S \ref{hida-subsec} for details). For the sake of simplicity, in this introduction we ignore the phenomenon of ``$p$-stabilization'' (see \S \ref{stabilization-subsec}). Thus, whenever we speak of an arithmetic invariant associated with $f_\wp$ we tacitly understand that this notion refers (at least when the weight of $f_\wp$ is larger than $2$) to the newform of level $N$ whose $p$-stabilization is equal to $f_\wp$, rather than to $f_\wp$ itself.

Let $\bar\Z$ denote the ring of integers in a fixed algebraic closure $\bar\Q$ of $\Q$ and let $\fP$ be a prime ideal of $\bar\Z$ above $p$. Let $\Q_{f_\wp}$ be the number field generated over $\Q$ by the Fourier coefficients of $f_\wp$, let $\cO_{f_\wp}$ be the ring of integers of $\Q_{f_\wp}$ and write $\Q_{f_\wp,\fP}$ (respectively, $\cO_{f_\wp,\fP}$) for the completion of $\Q_{f_\wp}$ (respectively, $\cO_{f_\wp}$) at the prime of $\cO_{f_\wp}$ under $\fP$. Moreover, denote by $V_{f_\wp,\fP}$ the $\fP$-adic Galois representation attached by Deligne to $f_\wp$, which is two-dimensional over $\Q_{f_\wp,\fP}$, let $V_{f_\wp,\fP}^\dagger$ be its self-dual twist and let $T^\dagger_{f_\wp,\fP}$ be a suitably chosen $\cO_{f_\wp,\fP}$-lattice inside $V_{f_\wp,\fP}^\dagger$. Let $\wp$ be an arithmetic prime of $\mathcal R$ of even weight $k_\wp>2$ and trivial character. Let $\Lambda_{f_\wp,\fP}(\Q)\subset H^1\bigl(\Q,T^\dagger_{f_\wp,\fP}\bigr)$ be the image of the $\fP$-adic \'etale Abel--Jacobi map attached to $f_\wp$ and denote by $\Sha_\fP(f_\wp/\Q)$ the $\fP$-primary Shafarevich--Tate group of $f_\wp$ over $\Q$ (see \S \ref{cycles-subsec} and \S \ref{sha-subsec}). The $\cO_{f_\wp,\fP}$-module $\Lambda_{f_\wp,\fP}(\Q)$ is finitely generated and we define the \emph{algebraic $\fP$-rank} $r_{\alg,\fP}(f_\wp)$ of $f_\wp$ to be the rank of $\Lambda_{f_\wp,\fP}(\Q)$ over $\cO_{f_\wp,\fP}$. In analogy with a well-known conjecture for abelian varieties over number fields, $\Sha_\fP(f_\wp/\Q)$ is expected to be finite for every $\wp$ as above (in fact, for any newform). 

One of the goals of this paper is to study the algebraic invariants $r_{\alg,\fP}(f_\wp)$ and $\Sha_\fP(f_\wp/\Q)$ as $\wp$ runs over (a suitable subset of) the arithmetic primes of $\mathcal R$ of even weight and trivial character. On the one hand, as remarked below, analytic arguments suggest that $r_{\alg,\fP}(f_\wp)$ should be identically equal either to $0$ or to $1$, except for at most finitely many $\wp$. On the other hand, as far as Shafarevich--Tate groups are concerned, a result that one would ideally like to prove is the following: $\Sha_\fP(f_\wp/\Q)$ is finite for all $\wp$ as soon as $\Sha_\fP(f_\wp/\Q)$ is finite for one $\wp$. Unfortunately, a result of this form seems to be out of reach of current techniques. In general, while it is difficult to approach the variations of $r_{\alg,\fP}(f_\wp)$ and $\Sha_\fP(f_\wp/\Q)$ separately, a joint study of these two invariants can be much more effective and rewarding.

To describe our results, we introduce some notation. Let $f$ be a newform of weight $2$, square-free level $N$ and trivial character and let $p$ be a \emph{good ordinary} prime number for $f$, that is, a prime such that $p\nmid N$ and $f$ is $p$-ordinary in the sense that $p$ does not divide the $p$-th Fourier coefficient of $f$. Here we are implicitly assuming that $f$ is also $\fP$-ordinary (see \S \ref{stabilization-subsec}), which is an important but technical point: in due course, we will carefully explain how to choose a suitable prime $\fP$ of $\bar\Z$ above $p$ (\emph{i.e.}, an embedding $\bar\Q\hookrightarrow\bar\Q_p$). Let $\hf\in\mathcal R[\![q]\!]$ be the $p$-adic Hida family of tame level $N$ passing through $f$, whose specializations will be denoted, as above, by $f_\wp$. Finally, write $A_f$ for the abelian variety over $\Q$ attached to $f$ via Shimura's construction and let $\Sha_{p^\infty}(A_f/\Q)$ be the $p$-primary part of the Shafarevich--Tate group of $A_f$ over $\Q$.  

Our first main result (Theorems \ref{sha-thm} and \ref{sha-thm2}) can be stated as follows.

\begin{theoremA} 
Suppose that $A_f$ is an elliptic curve and the rank of $A_f(\Q)$ is $0$ (respectively, $1$). For all but finitely many primes $p$ that are good ordinary for $f$, if $\Sha_{p^\infty}(A_f/\Q)$ is finite, then all specializations $f_\wp$ of $\hf$ of weight congruent to $2$ modulo $2(p-1)$ and trivial character satisfy $r_{\alg,\fP}(f_\wp)=0$ (respectively, $r_{\alg,\fP}(f_\wp)=1$) and $\#\Sha_\fP(f_\wp/\Q)<\infty$. 
\end{theoremA}
 
Now we turn our attention to an invariant of an analytic nature. Let $\f$ be a Hida family. For any specialization $f_\wp$ of $\f$ of even weight $k_\wp$ and trivial character, let $\varepsilon(f_\wp)\in\{\pm1\}$ be the root number of $f_\wp$, \emph{i.e.}, the sign in the functional equation for the $L$-function $L(f_\wp,s)$ of $f_\wp$. The root number controls the parity of the \emph{analytic rank} $r_\an(f_\wp)$ of $f_\wp$, \emph{i.e.}, the order of vanishing of $L(f_\wp,s)$ at $s=k_\wp/2$; namely, $\varepsilon(f_\wp)=(-1)^{r_\an(f_\wp)}$. It is known that $\varepsilon(f_\wp)$ is constant, except for finitely many $\wp$ that have necessarily weight $2$ and have been described by Mazur--Tate--Teitelbaum, when $\wp$ varies over the arithmetic primes of $\mathcal R$ as above. A prime outside this finite exceptional set will be called \emph{generic}; we set $\varepsilon(\f)\defeq\varepsilon(f_\wp)$ for any generic prime $\wp$ of $\mathcal R$ and call $\varepsilon(\f)$ the \emph{root number} of $\f$. It is convenient to define the \emph{minimal admissible generic rank} of $\f$ as
\[ r_{\min}(\f)\defeq\frac{1-\varepsilon(\f)}{2}. \] 
Equivalently, $r_{\min}(\f)=0$ if $\varepsilon(\f)=1$ and $r_{\min}(\f)=1$ if $\varepsilon(\f)=-1$.

A conjecture of Greenberg (\cite{Greenberg-CRM}) predicts that the analytic ranks of even weight modular forms in $\f$ should be equal to $r_{\min}(\f)$, with at most finitely many exceptions. Relatively little is known about this conjecture: as pointed out in \S \ref{greenberg-subsec}, the results that are currently available deal (under some technical assumptions on $\f$) either with $\varepsilon(\f)=-1$ and weight $2$ forms or with arbitrary (even) weight forms but $\varepsilon(\f)=1$. Observe that combining the conjecture of Greenberg with the conjectures of Birch--Swinerton-Dyer (in weight $2$) and of Beilinson--Bloch--Kato (in higher weight) on $L$-functions of modular forms justifies the expectation (Conjecture \ref{main-conj}) that $r_{\alg,\fP}(f_\wp)$ should equal $r_{\min}(\f)$ for all but finitely many $\wp$. It turns out (Corollaries \ref{sha-coro} and \ref{sha-coro2}) that Theorem A is consistent with (and gives partial evidence for) this conjectural statement.
 
Let the Hida family $\hf$ be as in Theorem A. As a second contribution (Theorem \ref{main-thm}), with no restriction on the dimension of $A_f$ but assuming the non-degeneracy of certain height pairings that have been introduced (following Gillet--Soul\'e) by S.-W. Zhang in \cite{Zhang-heights} to prove a counterpart for higher (even) weight modular forms of the Gross--Zagier formula, we offer

\begin{theoremB} 
Suppose that $f$ has analytic rank $1$ and that the height pairing in Zhang's formula is non-degenerate. For all but finitely many primes $p$ that are good ordinary for $f$, all specializations of $\hf$ of weight congruent to $2$ modulo $2(p-1)$ and trivial character have analytic rank $1$.
\end{theoremB}

Since the assumption on $r_\an(f)$ implies that $\varepsilon(\hf)=-1$, this result (albeit conditional on the non-degeneracy of Zhang's heights) provides supporting evidence (the first of this kind, to the best of our knowledge) for Greenberg's conjecture in weight larger than $2$ when $r_{\min}(\hf)=1$. It is worth remarking that the non-degeneracy that we need to impose is, in fact, predicted by the arithmetic analogues of the standard conjectures that have been proposed by Gillet and Soul\'e (\cite{GS-2}).

For the reader's convenience, we sketch our strategy for proving Theorems A and B. Under the assumption that $r_\an(f)\in\{0,1\}$, we introduce three sets $\Xi_f$, $\Theta_f$, $\Omega_f$ of prime numbers (\S \ref{heegner-subsec}, \S \ref{sha-subsec}, \S \ref{choice-subsec2}). Each of these sets, which are defined (using the Gross--Zagier formula and analytic results of Waldspurger, Bump--Friedberg--Hoffstein and Murty--Murty) in terms of, among other conditions, the non-triviality modulo $p$ of the imaginary quadratic Heegner point on $A_f$ appearing in the Gross--Zagier formula, consists of all but finitely many primes that are good ordinary for $f$. In particular, if $p$ belongs to any of the aforementioned sets and $\wp$ is an arithmetic prime of weight $k_\wp\equiv2\pmod{2(p-1)}$ and trivial character, then, thanks to results of Fischman on the image of $\Lambda$-adic Galois representations (\cite{Fischman}), the residual $\fP$-adic representation attached to $f_\wp$ has non-solvable image (\S \ref{lambda-subsec}; notice that the assumption that $N$ is square-free makes it easier to apply Fischman's results). This property, combined with work of Castella (\cite{CasHeeg}) and of Ota (\cite{Ota-JNT}) on the specializations of Howard's big Heegner points, leads us to our key technical result: for any $\wp$ as above, the imaginary quadratic Heegner cycle $y_\wp$ (to be denoted by $y_{\wp,K}$ in the main body of the paper) that was originally defined by Nekov\'a\v{r} is non-torsion over $\cO_{f_\wp,\fP}$ in the relevant \'etale Abel--Jacobi image (\S \ref{big-subsec}). Once the non-degeneracy of Zhang's heights is assumed, Theorem B for all $p\in\Xi_f$ is then a consequence (\S \ref{Q-subsec}) of Zhang's formula of Gross--Zagier type for modular forms (\S \ref{zhang-subsec}). 

Finally, suppose that $A_f$ is an elliptic curve. In order to prove Theorem A, we note that the assumption that $r\defeq\rank_\Z A_f(\Q)\in\{0,1\}$ and $\#\Sha_{p^\infty}(A_f/\Q)<\infty$ amounts, thanks to converses to the Kolyvagin--Gross--Zagier theorem due to Skinner--Urban (if $r=0$, \cite{SU}) and to W. Zhang (if $r=1$, \cite{zhang-selmer}), to the condition $r_\an(f)=r$. Since $y_\wp$ is not torsion, Theorem A for all $p\in\Omega_f$ if $r=0$ or for all $p\in\Theta_f$ if $r=1$ follows (\S \ref{sha-subsec2} and \S \ref{sha-subsec}) from Nekov\'a\v{r}'s results on the arithmetic of Chow groups of Kuga--Sato varieties (\cite{Nek}) combined with a comparison of \'etale Abel--Jacobi images over $\Q$ and over certain imaginary quadratic fields (\S \ref{AJ-subsec}), which may be interesting in its own right.

\subsection{Notation and conventions} \label{notation-subsec}

We denote by $\bar\Q$ an algebraic closure of $\Q$ and write $\bar\Z$ for the ring of integers in $\bar\Q$ (\emph{i.e.}, the integral closure of $\Z$ in $\bar\Q$). For every prime number $\ell$ we fix an algebraic closure $\bar\Q_\ell$ of $\Q_\ell$. Moreover, for every prime $\ell$ and every number field $F$ we also fix field embeddings
\[ \iota_\ell:\bar\Q\longmono\bar\Q_\ell,\quad\iota_F:F\longmono\bar\Q. \]
Later on, we will specify how to choose $\bar\Q_p$ for $p$ in any of the sets of primes $\Xi_f$, $\Theta _f$, $\Omega_f$ alluded to before, and then $\bar\Q$ will be the algebraic closure of $\Q$ in $\bar\Q_p$.

The map $\iota_\ell$ determines a prime ideal $\mathfrak L$ of $\bar\Z$ above $\ell$ that, in turn, induces a prime $\mathfrak L_F\defeq\iota_F^{-1}\bigl(\mathfrak L\cap\iota_F(F)\bigr)$ of $F$ above $\ell$. In order to simplify our notation, when there is no risk of confusion we will often use alternative symbols to denote the ideal $\mathfrak L_F$ and related objects. For example, we write $F_\mathfrak L$ in place of $F_{\mathfrak L_F}$ for the completion of $F$ at the prime $\mathfrak L_F$.

For any number field $K$ we denote by $G_K\defeq\Gal(\bar K/K)$ the absolute Galois group of $K$, where $\bar K$ is a fixed algebraic closure of $K$. For any continuous $G_K$-module $M$ we write $H^i(K,M)$ for the $i$-th continuous cohomology group of $G_K$ with coefficients in $M$ in the sense of Tate (\cite[\S 2]{Tate}). Finally, if $K/F$ is an extension of number fields, then  
\[ \res_{K/F}:H^i(F,M)\longrightarrow H^i(K,M),\quad\cores_{K/F}:H^i(K,M)\longrightarrow H^i(F,M) \] 
denote the restriction and corestriction maps in cohomology, respectively.

\subsection*{Acknowledgements} 

It is a pleasure to thank Matteo Longo and Rodolfo Venerucci for enlightening conversations on some of the topics of this paper. I would also like to express my gratitude to Maria Rosaria Pati for her very careful reading of this article and for pointing out several inaccuracies in a previous version of it. Finally, I wish to thank the anonymous referee for helpful remarks and suggestions.

\section{Galois representations and Hida families} \label{hida-sec}

In this section, we provide some background on Galois representations attached to Hecke eigenforms and on Hida families of modular forms; this will also give us an occasion to introduce notation that will be used throughout this paper. The reader who is conversant with these topics may wish to skim through this section and come back to it only if the need arises.

\subsection{Galois representations attached to modular forms} \label{cohomological-subsec}

Let $f\in S_k(\Gamma_0(M),\chi)$ be a normalized eigenform of weight $k\geq2$, level $M\geq3$ and character $\chi$, whose $q$-expansion will be denoted by
\[ f(q)=\sum_{n\geq1}a_n(f)q^n. \]
Let $\Q_f\defeq\Q\bigl(a_n(f)\mid n\geq1\bigr)$ be the Hecke field of $f$, \emph{i.e.}, the subfield of $\C$ that is generated over $\Q$ by the Fourier coefficients of $f$. It is well known that $\Q_f$ is a number field and that $a_n(f)$ is an algebraic integer for all $n\geq1$; let $\cO_{\Q_f}$ be the ring of integers of $\Q_f$. The Fourier coefficients of $f$ generate an order $\cO_f$ in $\cO_{\Q_f}$. Let $p$ be a prime number and fix a prime $\p$ of $\Q_f$ above $p$; by a slight abuse of notation, we write $\p$ also for the prime (actually, maximal) ideal of $\cO_f$ under $\p$. Deligne has attached to $f$ a $\p$-adic representation $V_{f,\p}$ of $G_\Q$ (\cite{Del-Bourbaki}), which is $2$-dimensional over the completion $\Q_{f,\p}$ of $\Q_f$ at $\p$, unramified outside $Mp$ and, by a result of Ribet (\cite[Theorem 2.3]{ribet}), irreducible.

Two ``normalizations'' of Deligne's representation are naturally available; in order to avoid any ambiguity, we specify the one we use in this paper. Denote by $\mathfrak H_k(\Gamma_1(M))$ the Hecke algebra over $\Z$ acting on $S_k(\Gamma_1(M))$. The eigenform $f$ comes equipped with a $\Q$-algebra homomorphism 
\[ \mathfrak H_k\bigl(\Gamma_1(M)\bigr)\otimes_\Z\Q\longrightarrow\C \]
whose image is $\Q_f$, hence $\Q_f$ is naturally endowed with an $\mathfrak H_k(\Gamma_1(M))$-algebra structure. Following Deligne's construction, and in light of Nekov\'a\v{r}'s definition of Heegner cycles (\cite{Nek}), we set
\begin{equation} \label{V-eq}
V_{f,\p}\defeq H^1_{\text{\'et}}\bigl(X_1(M)\otimes\bar\Q,j_*\Sym^{k-2}R^1\gamma_*\Q_p\bigr)\otimes_{\mathfrak H_k(\Gamma_1(M))\otimes_\Z\Q_p}\Q_{f,\p}, 
\end{equation} 
where $\gamma:\mathcal E\rightarrow Y_1(M)$ is the universal elliptic curve and $j:Y_1(M)\hookrightarrow X_1(M)$ is the open immersion that gives the canonical compactification of the modular curve $Y_1(M)$. The two-dimensional $\Q_{f,\p}$-vector space $V_{f,\p}$ is endowed with a natural action of $G_\Q$ and we write
\[ \rho_{f,\p}:G_\Q\longrightarrow\GL(V_{f,\p})\simeq\GL_2(\Q_{f,\p}) \] 
for the corresponding homomorphism. We call $V_{f,\p}$ the \emph{cohomological} realization of Deligne's representation: it is characterized by the fact that the characteristic polynomial of a \emph{geometric} Frobenius at a prime $\ell\nmid Mp$ is the Hecke polynomial
\begin{equation} \label{hecke-eq}
X^2-a_\ell(f)X+\chi(\ell)\ell^{k-1}. 
\end{equation}
For more details see, \emph{e.g.}, \cite[\S 2.3]{Ota-JNT}, \cite[Section 4]{saito-intro}, \cite[Theorem 1.2.4]{scholl}. By a suitably twisted Poincar\'e duality on (compactified) Kuga--Sato varieties, $V_{f,\p}$ is equipped with a $G_\Q$-equivariant, alternating, non-degenerate pairing
\[ V_{f,\p}\times V_{f,\p}\longrightarrow\Q_{f,\p}(1-k)\otimes_{\Q_{f,\p}}[\chi], \]
where $\Q_{f,\p}(1-k)$ denotes, as usual, a Tate twist of the trivial representation and $[\chi]$ is the one-dimensional representation induced by $\chi$ as in \cite[(1.3.4)]{NP}. This shows that if $V^*_{f,\p}\defeq\Hom_{\Q_{f,\p}}(V_{f,\p},\Q_{f,\p})$ is the dual (\emph{i.e.}, contragredient) representation of $V_{f,\p}$, then
\begin{equation} \label{dual-eq}
V^*_{f,\p}\simeq V_{f,\p}(k-1)\otimes_{\Q_{f,\p}}\bigl[\chi^{-1}\bigr]. 
\end{equation}
It follows that if $\chi$ is trivial (hence $k$ is even), which will eventually be the case in this paper, then $V_{f,\p}^\dagger\defeq V_{f,\p}(k/2)$ is self-dual in the sense that $V_{f,\p}^\dagger\simeq\bigl(V_{f,\p}^\dagger\bigr)^*(1)$.

\begin{remark} \label{irreducible-rem}
Irreducibility of a given representation is preserved by tensorization with one-dimensional representations (see, \emph{e.g.}, \cite[Exercise 2.2.14, (2)]{kowalski}), so Tate twists of $V_{f,\p}$ are irreducible. In particular, $V_{f,\p}^\dagger$ is irreducible.
\end{remark}

Suppose now that $k=2$ and $\chi$ is trivial, and let $A_f$ be the abelian variety over $\Q$ attached to $f$ by the Eichler--Shimura construction (\emph{cf.} \cite[\S 7.5]{shimura}). The ring $\cO_f$ embeds into the ring $\End_\Q(A_f)$ of endomorphisms of $A_f$ defined over $\Q$; in fact, $\Q_f=\End_\Q(A_f)\otimes_\Z\Q$ and $A_f$ is an abelian variety of $\GL_2$-type (see, \emph{e.g.}, \cite[Corollary 4.2]{ribet-twists}).

Denote by $\Ta_\p(A_f)$ the $\p$-adic Tate module of $A_f$ and let $V_\p(A_f)\defeq\Ta_\p(A_f)\otimes\Q$ be the associated $\Q_{f,\p}$-linear representation of $G_\Q$. There is an identification
\[ V_{f,\p}=H^1_{\text{\'et}}(A_f,\Q_{f,\p})\simeq{V_\p(A_f)}^*, \]
where the isomorphism above follows by combining \cite[Theorem 15.1, (a)]{Milne-AV} with the $G_\Q$-equivariant splitting $\Ta_p(A_f)=\oplus_{\pi\mid p}\Ta_\pi(A_f)$, with $\pi$ varying over the primes of $\Q_f$ above $p$. Taking the self-dual twist, we obtain
\begin{equation} \label{V-2-eq}
V_{f,\p}^\dagger=V_{f,\p}(1)\simeq V_\p(A_f)^*(1)\simeq V_\p(A_f), 
\end{equation}
where the rightmost isomorphism, which we fix once and for all, is a consequence of the Weil pairing.

Now let us go back to the general case of weight $k\geq2$. We call the dual $V_{f,\p}^*$ of $V_{f,\p}$ the \emph{homological} realization of Deligne's representation: the characteristic polynomial of an \emph{arithmetic} Frobenius at a prime $\ell\nmid Mp$ is the Hecke polynomial \eqref{hecke-eq}. 

\begin{remark} \label{irreducibility-rem}
Since $V_{f,\p}$ is finite dimensional, the irreducibility of $V_{f,\p}$ is equivalent to the irreducibility of $V_{f,\p}^*$ (\cite[Proposition 2.2.18, (4)]{kowalski}).
\end{remark}

Notice that the self-dual twist of $V_{f,\p}^*$ is ${(V_{f,\p}^*)}^\dagger\defeq V_{f,\p}^*(1-k/2)$. With this normalization, in weight $2$ we recover the Tate module:
\[ V^*_{f,\p}={H^1_{\text{\'et}}(A_f,\Q_{f,\p})}^*\simeq{V_\p(A_f)}. \] 

\begin{remark}
Using \eqref{dual-eq}, it is straightforward to check that if $\chi$ is trivial, then $V_{f,\p}^\dagger\simeq{(V_{f,\p}^*)}^\dagger$.
\end{remark}

\begin{remark}
If $k=2$, then ${(V_{f,\p}^*)}^\dagger=V_{f,\p}^*$, in accord with the fact that $V_\p(A_f)$ is self-dual.
\end{remark}

In this article we always use the cohomological realization of Deligne's representation and simply refer to $V_{f,\p}$ in \eqref{V-eq} as \emph{the} $\p$-adic Galois representation attached to $f$.

\subsection{Reduction of Galois representations} \label{reduction-subsec}

Assume that $f$ has trivial character $\chi$, so $k$ must be even. Denote by $\cO_{\Q_f,\p}$ the completion of $\cO_{\Q_f}$ at $\p$, \emph{i.e.}, the valuation ring of $\Q_{f,\p}$. Fix a $G_\Q$-stable $\cO_{\Q_f,\p}$-lattice $T_{f,\p}\subset V_{f,\p}$. As is explained, \emph{e.g.}, in \cite[Definition 2.1]{Ota-JNT}, if $p\nmid6M$ (a condition we shall impose in due course), then we take $T_{f,\p}$ to be the lattice in $V_{f,\p}$ generated by the image of the integral \'etale cohomology
\[ H^1_{\text{\'et}}\bigl(X_1(M)\otimes\bar\Q,j_*\Sym^{k-2}R^1\gamma_*\Z_p\bigr)\otimes_{\mathfrak H_k(\Gamma_1(M))\otimes_\Z\Z_p}\cO_{\Q_f,\p}. \]
Set
\[ T^\dagger_{f,\p}\defeq T_{f,\p}\otimes_{\cO_{\Q_f,\p}}\!\cO_{\Q_f,\p}(k/2)\subset V^\dagger_{f,\p} \]
for the corresponding Tate twist. When $k=2$ one chooses $T_{f,\p}$ so that $T^\dagger_{f,\p}$ corresponds to $\Ta_\p(A_f)$ under isomorphism \eqref{V-2-eq}. Let $\pi_{f,\p}$ be a uniformizer for $\cO_{\Q_f,\p}$ and define 
\[ \bar T_{f,\p}\defeq T_{f,\p}/\pi_{f,\p}T_{f,\p},\quad\bar T^\dagger_{f,\p}\defeq T^\dagger_{f,\p}\big/\pi_{f,\p} T^\dagger_{f,\p}. \]
These are two-dimensional representations of $G_\Q$ over the residue field $\F_\p\defeq\cO_{\Q_f,\p}/\pi_{f,\p}\cO_{\Q_f,\p}$ of $\Q_{f,\p}$ and we write
\[ \bar\rho_{f,\p}:G_\Q\longrightarrow\GL(\bar T_{f,\p}),\quad\bar\rho^\dagger_{f,\p}:G_\Q\longrightarrow\GL\bigl(\bar T^\dagger_{f,\p}\bigr) \]
for the corresponding homomorphisms. The reduced representations $\bar\rho_{f,\p}$ and $\bar\rho^\dagger_{f,\p}$ depend on the choice of an $\cO_{\Q_f,\p}$-lattice inside $V_{f,\p}$ (although the notation that we use does not reflect this dependence), but their semisimplifications
\[ \bar\rho^{\,ss}_{f,\p}:G_\Q\longrightarrow\GL(\bar T^{\,ss}_{f,\p}),\quad\bar\rho^{\dagger,ss}_{f,\p}:G_\Q\longrightarrow\GL\bigl(\bar T^{\dagger,ss}_{f,\p}\bigr) \]
do not. The representation $\bar\rho^{\,ss}_{f,\p}$ is the \emph{residual representation of $f$ at $\p$}. 

\begin{remark} \label{absolutely-rem}
Since $p\not=2$, the representation $\bar\rho_{f,\p}$ is irreducible if and only if it is absolutely irreducible (see, \emph{e.g.}, \cite[(1.5.3), (3)]{NP}).
\end{remark}

Let $\varepsilon_{\cyc}:G_\Q\rightarrow\Z_p^\times$ be the $p$-adic cyclotomic character and write
\[ \bar\varepsilon_{\cyc}:G_\Q\longrightarrow\F_p^\times \]
for the mod $p$ cyclotomic character, \emph{i.e.}, the reduction of $\varepsilon_{\cyc}$ modulo $p$. It follows that 
\begin{equation} \label{dagger-twist-eq}
\bar T^\dagger_{f,\p}=\bar T_{f,\p}(k/2), 
\end{equation}
where ``$(k/2)$'' means that the action of $G_\Q$ on $\bar T_{f,\p}$ is twisted by the $k/2$-th power of $\bar\varepsilon_{\cyc}$. In particular, $\bar T_{f,\p}$ is irreducible if and only if $\bar T^\dagger_{f,\p}$ is (see Remark \ref{irreducible-rem}).

\begin{remark}
As a consequence of \eqref{dagger-twist-eq}, there is an identification $\bar T^{\dagger,ss}_{f,\p}=\bar T^{\,ss}_{f,\p}(k/2)$.
\end{remark}

\begin{remark} \label{twist-rem}
If $k\equiv2\pmod{2(p-1)}$, then $\bar T^\dagger_{f,\p}=\bar T_{f,\p}(1)$.
\end{remark}

\subsection{Reduction of dual representations} \label{residual-dual-subsec}

As in \S \ref{reduction-subsec}, assume that the character of $f$ is trivial. The lattice $T_{f,\p}\subset V_{f,\p}$ induces a \emph{dual lattice} $T_{f,\p}^*\defeq\Hom_{\cO_{\Q_f,\p}}(T_{f,\p},\cO_{\Q_f,\p})$ inside $V^*_{f,\p}$. Clearly, $T_{f,\p}^*$ is stable under the action of $G_\Q$. Set
\[ \overline{T_{f,\p}^*}\defeq T_{f,\p}^*\big/\pi_{f,\p}T_{f,\p}^*, \]
which is a two-dimensional representation of $G_\Q$ over $\F_\p$. Since $T_{f,\p}$ is free (hence projective) over $\cO_{\Q_f,\p}$, there is a natural identification
\begin{equation} \label{reduction-dual-eq}
\overline{T_{f,\p}^*}=\Hom_{\F_\p}\bigl(\bar T_{f,\p},\F_\p\bigr)\eqdef\bigl(\bar T_{f,\p}\big)^*. 
\end{equation}
We shall simply write $\bar T_{f,\p}^*$ for the $\F_\p$-linear representation of $G_\Q$ in \eqref{reduction-dual-eq}: it can equivalently be interpreted either as the reduction of $T_{f,\p}^*$ or as the dual of $\bar T_{f,\p}$. We denote by
\[ \bar\rho_{f,\p}^*:G_\Q\longrightarrow\GL(\bar T^*_{f,\p}) \]
the corresponding homomorphism. 

\begin{remark}
As was already pointed out for $V_{f,\p}$ in Remark \ref{irreducibility-rem}, $\bar T_{f,\p}$ is irreducible if and only if $\bar T^*_{f,\p}$ is.
\end{remark}

\begin{remark} \label{dual-twist-rem}
It follows from \eqref{dual-eq} that if $k\equiv2\pmod{2(p-1)}$, then $\bar T_{f,\p}^*\simeq\bar T_{f,\p}(1)$.
\end{remark}

\subsection{$p$-stabilization of modular forms} \label{stabilization-subsec}

Let $g\in S_k(\Gamma_0(M),\chi)$ be a normalized eigenform of weight $k\geq2$ and character $\chi$, with $q$-expansion $g(q)=\sum_{n\geq1}a_n(g)q^n$ and Hecke field $\Q_g$. Let $\cO_{\Q_g}$ be the ring of integers of $\Q_g$, let $p$ be a prime number such that $p\nmid M$, consider the prime $\fP_g\defeq\fP\cap\cO_{\Q_g}$ of $\Q_g$ above $p$ that was introduced in \S \ref{notation-subsec} and denote by $\Q_{g,\fP}$ the completion of $\Q_g$ at $\fP_g$. Assume that $g$ is \emph{$\fP$-ordinary}, \emph{i.e.}, $a_p(g)$ is a unit of the local field $\Q_{g,\fP}$; equivalently, $a_p(g)\notin\fP_g$. It follows that the Hecke polynomial
\[ X^2-a_p(g)X+\chi(p)p^{k-1}=(X-\alpha)(X-\beta) \]
has exactly one root that is a $\fP_g$-adic unit, say $\alpha$. If $g$ is a newform, then the \emph{$p$-stabilization} of $g$ is defined to be 
\[ g^\sharp(z)\defeq g(z)-\beta g(pz). \]
Then $g^\sharp$ is a $\fP$-ordinary normalized eigenform of weight $k$ and level $Mp$ such that, with self-explaining notation, $a_\ell(g^\sharp)=a_\ell(g)$ for all primes $\ell\neq p$ and $U_p(g^\sharp)=\alpha g^\sharp$, where $U_p$ is the usual Hecke operator on modular forms.

\begin{lemma} \label{L-functions-lemma}
$\ord_{s=k/2}L(g,s)=\ord_{s=k/2}L(g^\sharp,s)$.
\end{lemma}

\begin{proof} The $L$-function of $g^\sharp$ is equal to the $L$-function of $g$ with the Euler factor ${(1-\beta p^{-s})}^{-1}$ removed, and the claim follows. \end{proof}

With a terminology that will be introduced in Definition \ref{analytic-rank-def}, Lemma \ref{L-functions-lemma} says that the analytic ranks of $g$ and of $g^\sharp$ are equal.

\begin{remark} \label{isom-rep-rem}
With notation as above, let $V_{g,\fP}$ be the $\fP_g$-adic representation of $G_\Q$ attached to $g$ (see \S \ref{cohomological-subsec}). Recall that $V_{g,\fP}$ is irreducible, unramified outside $Mp$ and with the property that if $\ell$ is a prime such that $\ell\nmid Mp$, then the trace of a Frobenius element at $\ell$ acting on $V_{g,\fP}$ is $a_\ell(g)$. Since semisimple Galois representations are determined (up to isomorphism) by the traces of Frobenius elements at all but finitely many (unramified) primes (see, \emph{e.g.}, \cite[Lemme 3.2]{DS}), we conclude that $V_{g,\fP}$ and $V_{g^\sharp,\fP}$ (or, rather, their base changes to $\bar\Q_p$) are equivalent. 
\end{remark}

\begin{remark}
A closely related (but not equivalent) notion is that of a \emph{$p$-ordinary} form $g$, by which we mean that $a_p(g)\notin p\cO_{\Q_g}$. Of course, since $p\cO_{\Q_g}\subset\fP_g$, if $g$ is $\fP$-ordinary, then $g$ is $p$-ordinary; however, $g$ being $p$-ordinary does not guarantee, in general, that there is a prime ideal $\mathcal P$ of $\cO_{\Q_g}$ such that $g$ is $\mathcal P$-ordinary. At the end of \S \ref{heegner-subsec} we shall explain how to deal, in our context, with this issue.
\end{remark}

Our next goal is to introduce, as in \cite[Definition 2.5]{GS}, the notion of $p$-stabilized newform. A $\fP$-ordinary, normalized eigenform $g\in S_k\bigl(\Gamma_0(Mp^r),\chi\bigr)$ is a \emph{$p$-stabilized newform} (of tame conductor $M$) if the following two conditions hold:
\begin{enumerate}
\item the conductor of $g$ is divisible by $M$;
\item the level of $g$ is divisible by $p$.
\end{enumerate}
It can be checked that a $p$-stabilized newform $g$ is either already a newform of level $Mp^r$ or is the $p$-stabilization $f^\sharp$ of a newform $f$ of level $M$ as defined above. As previously remarked, in the latter case the level of $g$ is $Mp$. Furthermore, as a consequence of the strong multiplicity one theorem (see, \emph{e.g.}, \cite[Theorem 3.22]{hida-modular}), a normalized newform $f$ as above is unique.

\begin{remark} \label{old-rem}
Let $g\in S_k(\Gamma_0(Mp^r))$, with $r\geq1$, be a normalized eigenform of weight $k\geq2$. As a consequence of \cite[Proposition 3.1]{hida-measure}, if $g$ is $\fP$-ordinary and $k>2$, then $g$ is old at $p$ (see also \cite[Lemma 2.1.5]{Howard-Inv}). It follows that if $g$ is a $p$-stabilized newform of weight $k>2$, then necessarily $g=f^\sharp$ for a newform $f$ of level $M$.
\end{remark}

For further details and related results, see \cite[Theorem 2.6]{GS}.

\subsection{Hida families of modular forms} \label{hida-subsec}

We recall basic notions of Hida's theory of $p$-adic families of modular forms. For details and proofs, see \cite{hida86b}, \cite{hida86a}, \cite[Chapter 7]{hida-elementary}.

Let $N\geq1$ be an integer and let $p$ be a prime number such that $p\nmid N$. Set $\Gamma\defeq1+p\Z_p\subset\Z_p^\times$. For our purposes, a ($p$-adic) \emph{Hida family} (of tame level $N$) consists of
\begin{itemize}
\item a complete local noetherian domain $\mathcal R$ that is finitely generated and flat as a module over the Iwasawa algebra $\cO[\![\Gamma]\!]\simeq\cO[\![T]\!]$, where $\cO$ is a suitable finite extension of $\Z_p$ to be specified in due course;
\item a (dense) collection of distinguished points
\[ \mathcal X^{\text{arith}}\subset\Hom_{\text{cont}}(\mathcal R,\bar\Q_p) \]
called \emph{arithmetic morphisms};
\item a formal $q$-expansion $\f=\sum_{n\geq1}a_n(\f)q^n\in\mathcal R[\![q]\!]$
\end{itemize}
such that for all $\eta\in\mathcal X^{\text{arith}}$ the power series 
\[ f_\eta\defeq\sum_{n\geq1}\eta\bigl(a_n(\f)\bigr)q^n\in\bar\Q_p[\![q]\!] \]
is the $q$-expansion of an ordinary cuspidal eigenform on $\Gamma_1(Np^r)$ for some $r=r_\wp\geq1$. The modular form $f_\eta$ is called the \emph{specialization of $\f$ at $\eta$}. By definition, a continuous homomorphism $\eta:\mathcal R\rightarrow\bar\Q_p$ is \emph{arithmetic} if the composition 
\[ \Gamma\longrightarrow\mathcal R^\times\overset\eta\longrightarrow\bar\Q_p^\times \]
of $\eta$ with the canonical map $\Gamma\rightarrow\mathcal R^\times$ has the form 
\begin{equation} \label{arithmetic-eq}
\gamma\longmapsto\psi(\gamma)\gamma^{k-2}
\end{equation}
for some integer $k\geq2$ and some finite order character $\psi$ of $\Gamma$. The kernels of arithmetic morphisms are called \emph{arithmetic primes} of $\mathcal R$. The integer $k$ in \eqref{arithmetic-eq} is the \emph{weight} of the arithmetic morphism (or of the associated arithmetic prime), while $\psi$ is its \emph{wild character}. If $\wp$ is an arithmetic prime of $\mathcal R$ of weight $k$, then we set
\[ k_\wp\defeq k,\quad\psi_\wp\defeq\psi. \]
Given an arithmetic prime $\wp$, we say that $\wp$ has \emph{trivial character} if its wild character $\psi_\wp$ is trivial; moreover, if $\wp$ corresponds to a morphism $\eta_\wp$, then we write $f_\wp$ for $f_{\eta_\wp}$, whose weight is $k_\wp$. The field $\mathcal F_\wp\defeq\mathcal R_\wp/\wp\mathcal R_\wp$ is a finite extension of $\Q_p$ to which the Fourier coefficients of $f_\wp$ belong; the quotient $\mathcal R/\wp$ is the valuation ring of $\mathcal F_\wp$. We fix an embedding $\iota_\wp:\mathcal F_\wp\hookrightarrow\bar\Q_p$. Finally, for each arithmetic $\wp$ the specialization $f_\wp$ is a $p$-stabilized newform (of tame conductor $M$) in the sense of \S \ref{stabilization-subsec} (\cite[Corollary 1.3]{hida86b}).

\begin{remark}
By a slight abuse of terminology, we will usually identify a Hida family with the formal $q$-expansion $\f\in\mathcal R[\![q]\!]$.
\end{remark}

By Hida theory (\cite[Theorem 2.1]{hida86b}), there is a ``big'' representation $\T=\T_{\f}$ of $G_\Q$ that (under standard assumptions on residual representations, \emph{cf.} \S \ref{residual-subsec}) is free of rank $2$ over $\mathcal R$ and satisfies the following property: for every arithmetic prime $\wp$ of $\mathcal R$ the quotient $\T_\wp/\wp\T_\wp$ is equivalent over $\bar\Q_p$ (\emph{i.e.}, after a finite base change) to the dual $V_{f_\wp,\fP}^*$ of the representation $V_{f_\wp,\fP}$ of $G_\Q$ attached to $f_\wp$ as in \S \ref{cohomological-subsec} (see, \emph{e.g.}, \cite[(1.5.5)]{NP}). The representation
\begin{equation} \label{hida-rep-eq}
\rho_{\f}:G_\Q\longrightarrow\GL(\T)\simeq\GL_2(\mathcal R) 
\end{equation}
is unramified outside $Np$ and 
\[ \tr\bigl(\rho_{\f}(\Frob_\ell)\bigr)=a_\ell(\f) \]
for all prime numbers $\ell\nmid Np$, where $\Frob_\ell$ denotes the conjugacy class in $G_\Q$ of an arithmetic Frobenius at $\ell$.

\begin{remark} \label{notation-rem}
From here on, in order to lighten our notation, for every modular form $g$ as above we simply write $V_g$ in place of $V_{g,\fP}$. Analogously, $T_{g,\fP}$ will be denoted by $T_g$.
\end{remark}

Now let $f\in S_k(\Gamma_1(M))$ be a $\fP$-ordinary normalized newform of weight $k\geq2$. Notation being as in \S \ref{stabilization-subsec}, set
\[ f_0\defeq\begin{cases}f&\text{if $p\,|\,M$},\\[3mm]f^\sharp&\text{if $p\nmid M$}. \end{cases} \]
The cusp form $f_0$ can be characterized as the unique (normalized) $\fP$-ordinary eigenform of weight $k$ and level divisible by $p$ with the property that $a_n(f_0)=a_n(f)$ except for those $n$ divisible by $p$ (\cite[Lemma 3.3]{hida-measure}). A fundamental result of Hida (\cite[Corollary 3.7]{hida86a}, \cite[\S 7.3, Theorem 3]{hida-elementary}) asserts that there is a unique Hida family $\f\in\mathcal R[\![q]\!]$ such that $f_0=f_\wp$ for some arithmetic prime $\wp$ of $\mathcal R$. This property is often expressed by saying that there is a unique Hida family \emph{passing through} $f$.

\section{Greenberg's conjecture for Hida families} \label{greenberg-sec}

After recalling some properties of root numbers in Hida families of modular forms, we state Greenberg's conjecture (Conjecture \ref{greenberg-conj}), which predicts that the analytic ranks of all but finitely many forms of even weight and trivial character in a Hida family should be as small as possible.

\subsection{Root numbers in Hida families} \label{root-numbers-subsec}

Let $g\in S_k(\Gamma_0(M))$ be a normalized newform of weight $k\geq2$ with $q$-expansion $g(q)=\sum_{n\geq1}a_n(g)q^n$ and write $L(g,s)$ for its $L$-function. The completed $L$-function of $g$ is $\Lambda(g,s)\defeq(2\pi)^{-s}\Gamma(s)M^{s/2}L(g,s)$, where $\Gamma(s)$ is the classical $\Gamma$-function. It is well known that $\Lambda(g,s)$ satisfies a functional equation 
\begin{equation} \label{L-functional-eq}
\Lambda(g,s)=\varepsilon(g)\Lambda(g,k-s) 
\end{equation}
where $\varepsilon(g)\in\{\pm1\}$ is the \emph{root number} of $g$ (see, \emph{e.g.}, \cite[Theorem 9.27]{Knapp}). 

\begin{definition} \label{analytic-rank-def}
The \emph{analytic rank} of $g$ is $r_\an(g)\defeq\ord_{k/2}L(g,s)$. 
\end{definition}

It follows that $r_\an(g)$ is even if $\varepsilon(g)=1$ and is odd if $\varepsilon(g)=-1$, \emph{i.e.}, ${(-1)}^{r_\an(g)}=\varepsilon(g)$.

\begin{remark} \label{analytic-rank-rem}
The definition of analytic rank makes sense also when the eigenform $g$ is not a newform (and so $g$ may not satisfy a functional equation of the shape \eqref{L-functional-eq}). Lemma \ref{L-functions-lemma} ensures that $r_\an(g)=r_\an(g^\sharp)$.
\end{remark}

\begin{remark}
For a description of the functional equation satisfied by a newform on $\Gamma_0(M)$ whose character is not necessarily trivial, the reader is referred, \emph{e.g.}, to \cite[Theorems 4.3.12 and 4.6.15]{Miyake}.
\end{remark}

Now let $\f\in\mathcal R[\![q]\!]$ be a Hida family as in \S \ref{hida-subsec} and let $\wp$ be an arithmetic prime of $\mathcal R$ of even weight and trivial character. By Remark \ref{old-rem}, if $k_\wp>2$, then there is a newform $f^\flat_\wp$ of level $N$ and trivial character such that $(f^\flat_\wp{)}^\sharp=f_\wp$. The same is true if $k_\wp=2$ and $f_\wp$ is not new at $p$, while if $f_\wp$ is a newform of level $Np^r$ for some $r\geq1$, then we set $f_\wp^\flat\defeq f_\wp$. For every $\wp$ as above, let $\iota_\wp:\mathcal F_\wp\hookrightarrow\bar\Q_p$ be the embedding that we fixed in \S \ref{hida-subsec}. Fix also an embedding $\iota_\infty:\bar\Q\hookrightarrow\C$. If $f_\wp(q)=\sum_{n\geq1}a_\wp(n)q^n$, then the set $\bigl\{\iota_\wp(a_\wp(n))\mid n\geq1\bigr\}$ generates a finite extension of $\Q$ inside $\bar\Q_p$, so $\Q\bigl(\iota_\wp(a_\wp(n))\mid n\geq1\bigr)\subset\bar\Q$. By identifying $f_\wp$ with $f_\wp^{\iota_\infty}(q)\defeq\sum_{n\geq1}\iota_\infty\bigl(\iota_\wp(a_\wp(n))\bigr)q^n$, we can view $f_\wp$ as a classical modular form with complex Fourier coefficients. Analogous considerations apply to $f_\wp^\flat$.

Except for finitely many arithmetic primes $\wp$ of $\mathcal R$, which were explicitly described by Mazur, Tate and Teitelbaum in \cite{MTT}, the root number $\varepsilon(f^\flat_\wp)$ of the newform $f_\wp^\flat$ is constant as $\wp$ varies (see, \emph{e.g.}, \cite[Proposition 12.7.14.4, (i)]{Nek-Selmer} or \cite[(3.4.4)]{NP} for details). An arithmetic prime not belonging to this finite set of bad (or, better, \emph{exceptional}) primes is called \emph{generic}; similarly, the specialization of $\f$ at a generic arithmetic prime is called \emph{generic}.

\begin{definition} \label{root-number-def}
Let $\wp$ be a generic arithmetic prime of $\mathcal R$. The integer $\varepsilon(\f)\defeq\varepsilon(f^\flat_\wp)$ is the \emph{root number of $\f$}. 
\end{definition}

The next lemma will be used to deduce the vanishing of special values of higher weight forms in Hida families.

\begin{lemma} \label{root-number-lemma}
If $\wp$ is an arithmetic prime of $\mathcal R$ such that $k_\wp>2$, then $\varepsilon(f^\flat_\wp)=\varepsilon(\f)$.
\end{lemma}

\begin{proof} As recorded in \cite[(3.1.1)]{NP}, an exceptional prime has always weight $2$. \end{proof}

It is convenient to introduce the following notion.

\begin{definition} \label{minimal-rank-def}
The \emph{minimal admissible generic rank} of $\f$ is
\[ r_{\min}(\f)\defeq\frac{1-\varepsilon(\f)}{2}. \] 
\end{definition}

Equivalently, the minimal admissible generic rank of $\f$ is the smallest analytic rank of a generic specialization of $\f$ that is allowed by the functional equation: $r_{\min}(\f)=0$ if $\varepsilon(\f)=1$ and $r_{\min}(\f)=1$ if $\varepsilon(\f)=-1$.

\subsection{Greenberg's conjecture} \label{greenberg-subsec}

The conjecture that we state below, which is concerned with the analytic rank of an even weight form in a Hida family $\f\in\mathcal R[\![q]\!]$ rather than merely with its root number, is essentially due to Greenberg (\cite{Greenberg-CRM}). 

\begin{conjecture}[Greenberg] \label{greenberg-conj}
The equality
\[ r_\an(f_\wp)=r_{\min}(\f) \]
holds for all but finitely many arithmetic primes $\wp$ of $\mathcal R$ of even weight.
\end{conjecture}

In other words, Conjecture \ref{greenberg-conj} predicts that the analytic ranks of even weight modular forms in a Hida family $\f$ should be as small as allowed by the functional equation, with at most finitely many exceptions. Notice that, in light of Lemma \ref{L-functions-lemma}, we can equivalently formulate Conjecture \ref{greenberg-conj} in terms of $f^\flat_\wp$.

\begin{remark}
Greenberg's conjectures on analytic ranks were proposed in \cite{Greenberg-CRM} in a somewhat more general form, so Conjecture \ref{greenberg-conj} should be viewed as a special case of them. The reader is referred to \cite[p. 101, Conjecture]{Greenberg-CRM} for the original conjecture for newforms of weight $2$ and to the discussion following it for a conjecture for newforms of arbitrary weight.
\end{remark}

\begin{remark} \label{p-adic-variant-rem}
Conjecture \ref{greenberg-conj} admits a $p$-adic variant in terms of $p$-adic $L$-functions of modular forms, after Mazur--Tate--Teitelbaum (see, \emph{e.g.}, \cite[p. 475]{NP}).
\end{remark}

Relatively little is known about Conjecture \ref{greenberg-conj} or the conjectures in \cite{Greenberg-CRM}. Here we would like to mention results by Greenberg (\cite{greenberg-BSD}, \cite{greenberg-critical}) and by Rohrlich (\cite{rohrlich-anticyc}, \cite{rohrlich-cyc}, \cite{rohrlich-division}), which naturally fit in with the setting of \cite{Greenberg-CRM}, but apply only to certain rather small subsets of newforms and deal mainly with weight $2$ and almost exclusively with analytic rank $0$. Directly germane to Conjecture \ref{greenberg-conj} is more recent work by Howard (\cite{Howard-derivatives}), who uses his theory of big Heegner points (\cite{Howard-Inv}) to show a result of the following type: if in a Hida family there is a weight $2$ form whose analytic rank is $1$, then all but finitely many weight $2$ forms in the family enjoy the same property (in fact, as a consequence of the existence of the Mazur--Kitagawa two-variable $p$-adic $L$-function, the same is true for analytic rank $0$ as well). We also remark that analytic rank $0$ results for infinitely many (or even all but finitely many) higher (even) weight forms in a Hida family are available (see, \emph{e.g.}, \cite[Corollary 4]{BD-rational} and \cite[Theorem 7]{Howard-derivatives}).

\begin{remark}
By exploiting the theory of quaternionic big Heegner points developed in \cite{LV-MM}, some of the results by Howard are extended in \cite{LV-Pisa} to more generale arithmetic contexts.
\end{remark}

Under some technical assumptions (among which, the non-degeneracy of certain height pairings \emph{\`a la} Gillet--Soul\'e between Heegner cycles), in this paper we prove a result (Theorem B of the introduction) in the direction of Conjecture \ref{greenberg-conj} when $r_{\min}(\f)=1$ and the prime $\wp$ has trivial character and weight $k_\wp>2$ such that $k_\wp\equiv2\pmod{2(p-1)}$. 



\section{Weight two newforms and Hida families} \label{weight-2-sec}

In this section, we consider a newform $f$ of weight $2$, square-free level $N$ and trivial character such that $r_\an(f)=1$. We introduce a set $\Xi_f$ that consists of all but finitely many prime numbers that are ordinary for $f$, then for every $p\in\Xi_f$ we explain how to choose a prime ideal $\fP$ of $\bar\Z$ above $p$ such that $f$ is $\fP$-ordinary. For every $p\in\Xi_f$, let $\hf$ be the $p$-adic Hida family passing through $f$, as in \S \ref{hida-subsec}. The families $\hf$, where $p$ varies in $\Xi_f$ (or, in the rank $0$ case, in an analogously defined set $\Omega_f$), are essentially those for which in later sections we shall prove our arithmetic results (Theorem A and Theorem B in the introduction). It is worth emphasizing that $\Xi_f$ is defined so that, for each $p\in\Xi_f$, the $\fP$-adic representations (as well as their reductions) attached to suitable specializations of $\hf$ have non-solvable images. The fact that one can choose $\Xi_f$ in such a way that this property is satisfied rests upon results of Fischman on the image of $\Lambda$-adic Galois representations (\cite{Fischman}).

\subsection{Base change to quadratic fields} \label{base-change-subsec}

As in \S \ref{root-numbers-subsec}, write $g\in S_k(\Gamma_0(M))$ for a normalized newform with $q$-expansion $g(q)=\sum_{n\geq1}a_n(g)q^n$. Let $K$ be a quadratic field (later on, we will again take $K$ to be imaginary) and let $\chi_K$ be the Dirichlet character associated with $K$. The $L$-function of $g$ over $K$ is defined to be the product
\begin{equation} \label{base-change-L-eq}
L(g/K,s)\defeq L(g,s)L(g\otimes\chi_K,s), 
\end{equation}
where $L(g,s)$ is, as usual, the $L$-function of $g$ and 
\[ L(g\otimes\chi_K,s)\defeq\sum_{n\geq1}\frac{a_n(g)\chi_K(n)}{n^s}. \]
Let $D_K$ be the discriminant of $K$ and set $\Lambda(g,\chi_K,s)\defeq(2\pi)^{-s}\Gamma(s)(D_K^2M)^{s/2}L(g\otimes\chi_K,s)$. As is pointed out, \emph{e.g.}, in \cite[p. 543]{BFH}, if $M$ and $D_K$ are coprime, \emph{i.e.}, if no prime factor of $M$ ramifies in $K$, then $\Lambda(g\otimes\chi_K,s)$ satisfies the functional equation
\[ \Lambda(g\otimes\chi_K,s)=\varepsilon(g)\chi_K(-M)\Lambda(g\otimes\chi_K,k-s) \]
where, as in \eqref{L-functional-eq}, $\varepsilon(g)$ is the root number of $g$. Define $\varepsilon(g,\chi_K)\defeq\varepsilon(g)\chi_K(-M)$. If every prime dividing $M$ splits in $K$, then $\chi_K(M)=1$, hence $\varepsilon(g,\chi_K)=\varepsilon(g)\chi_K(-1)$. It follows that
\begin{equation} \label{twisted-root-number-eq}
\varepsilon(g,\chi_K)=\begin{cases}\varepsilon(g)&\text{if $K$ is real},\\[3mm]-\varepsilon(g)&\text{if $K$ is imaginary}.\end{cases} 
\end{equation}
In the following definition, $K$ is a quadratic field such that $D_K$ and $M$ are coprime.

\begin{definition} \label{analytic-rank-K-def}
The \emph{analytic rank} of $g$ \emph{over $K$} is $r_\an(g/K)\defeq\ord_{k/2}L(g/K,s)$. 
\end{definition}

We immediately obtain

\begin{proposition} \label{analytic-rank-K-prop}
Let $K$ be an imaginary quadratic field in which all prime factors of $M$ split. Then $r_\an(g/K)\geq1$.
\end{proposition}

\begin{proof} Combine \eqref{base-change-L-eq} and \eqref{twisted-root-number-eq}. \end{proof}

\begin{remark}
Proposition \ref{analytic-rank-K-prop} is valid, more generally, for imaginary quadratic fields $K$ such that $(D_K,M)=1$ and the number of primes dividing $M$ that are inert in $K$ (counted with multiplicity) is even. This is the arithmetic setting that is considered, \emph{e.g.}, in \cite{LV-MM} and \cite{LV-Pisa} (on the proviso that if a prime $\ell$ is inert in $K$, then $\ell^2\nmid M$); it involves working with big Heegner points that are built out of Heegner points on Shimura curves attached to arbitrary (indefinite) quaternion algebras over $\Q$ (see, \emph{e.g.}, \cite[\S 6]{LV-MM} and \cite[\S 2]{LV-IJNT} for an introduction to Hida theory on quaternion algebras). Since in the present article we will ultimately be interested in results over $\Q$, there is no need for us to consider this more general situation here (notice, moreover, that the relevant specialization results are not available yet in the quaternionic case).
\end{remark}

\begin{remark}
The choice that we made in \eqref{base-change-L-eq} to take the product $L(g,s)L(g\otimes\chi_K,s)$ as the definition of $L(g/K,s)$ is perfectly adequate for our goals. Actually, one could also adopt a broader representation-theoretic definition for $L(g/K,s)$ as the $L$-function of the system of $\lambda$-adic Galois representations of $G_K$ attached to $g$, where $\lambda$ varies over all the primes of $\Q_g$; thanks to work of Deligne, Langlands and Carayol (see, \emph{e.g.}, \cite{carayol}), we know that this system is (strictly) compatible in the sense of Serre (\cite[Chapter I, \S 2.3]{serre-abelian}). From this perspective, equality \eqref{base-change-L-eq} is an instance of a formula for the $L$-function of $g$ over a finite abelian extension of $\Q$ (see, \emph{e.g.}, \cite[Proposition 3]{Brunault} for a precise statement in the case of elliptic curves).
\end{remark}

\subsection{The newform $f$ of weight $2$} \label{choice-subsec}

Here we introduce the newform $f$ that will be one of our basic data. Let $f\in S_2(\Gamma_0(N))$ be a normalized newform of weight $2$ and level $N$ with $q$-expansion $f(q)=\sum_{n\geq1}a_n(f)q^n$. Assume that 
\begin{itemize}
\item $N$ is square-free.
\end{itemize}
This condition guarantees (see, \emph{e.g.}, \cite[p. 34]{ribet}) that $f$ has no complex multiplication in the sense of \cite[p. 34, Definition]{ribet}. Equivalently, the abelian variety $A_f$ over $\Q$ attached to $f$ has no complex multiplication. Furthermore, since $N$ is square-free, $A_f$ is semistable, so all endomorphisms of $A$ are defined over $\Q$ (see, \emph{e.g.}, \cite[Corollary 1.4, (a)]{ribet-annals}). In particular, if $\End_{\bar\Q}(A_f)$ denotes the ring of \emph{all} endomorphisms of $A_f$, then $\Q_f=\End_{\bar\Q}(A_f)\otimes\Q$ and $\End_{\bar\Q}(A_f)$ is commutative.

Now write $\mathfrak H_2(\Gamma_0(N),\Z)^{\mathrm{red}}$ for the reduced Hecke algebra acting on $S_2(\Gamma_0(N))$. As in \cite[\S 2.3]{Fischman}, let $\Sigma_f$ be the set of prime numbers $\ell$ satisfying the following conditions:
\begin{itemize}
\item $\ell$ is ordinary for $f$ (\emph{i.e.}, $\ell\nmid a_\ell(f)$ in $\cO_{\Q_f}$);
\item $\ell$ does not belong to the finite set that is excluded by \cite[Theorem 2.1]{Fischman};
\item $\ell$ does not divide the discriminant of $\mathfrak H_2(\Gamma_0(N),\Z)^{\mathrm{red}}$;
\item $\ell$ does not divide $210N$.
\end{itemize}

From now until Section \ref{sha-sec2}, the following assumption will be in force.

\begin{assumption} \label{main-ass}
$r_\an(f)=1$.
\end{assumption}

In particular, if $\varepsilon(f)$ denotes, as in \eqref{L-functional-eq}, the root number of $f$, then $\varepsilon(f)=-1$.



\subsection{Heegner points and choice of $p$} \label{heegner-subsec}

Let $f\in S_2(\Gamma_0(N))$ be the newform that was chosen in \S \ref{choice-subsec}. The abelian variety $A_f$ is a quotient of the Jacobian of the (compact) modular curve $X_0(N)$ of level $\Gamma_0(N)$ and the dimension of $A_f$ is equal to the degree of $\Q_f$. As remarked in \S \ref{cohomological-subsec}, the ring of endomorphisms of $A_f$, which are all defined over $\Q$, contains (a subring isomorphic to) $\cO_f$. 

Recall that, by Assumption \ref{main-ass}, $r_\an(f)=1$, which implies that $\varepsilon(f)=-1$. By a result of Waldspurger (\cite{Waldspurger}), reproved by Bump--Friedberg--Hoffstein in a formulation that is more convenient for our purposes (\cite[p. 543, Theorem, (ii)]{BFH}), there exists an imaginary quadratic field $K$, whose associated Dirichlet character will be denoted (as in \S \ref{base-change-subsec}) by $\chi_K$, such that
\begin{itemize}
\item[(a)] the primes dividing $N$ split in $K$;
\item[(b)] $r_\an(f\otimes\chi_K)=0$.
\end{itemize}
Fix once and for all such a field $K$. In particular, condition (a) tells us that $K$ satisfies the \emph{Heegner hypothesis} relative to $N$. The theory of complex multiplication allows one to define a systematic supply of Heegner points $\alpha_c\in A_f(K_c)$ indexed by integers $c\geq1$ coprime to $N$, where $K_c$ is the ring class field of $K$ of conductor $c$ (in particular, $K_1$ is the Hilbert class field of $K$); these points arise by modularity from Heegner points $x_c\in X_0(N)(K_c)$. It is well known that the points $\alpha_c$ form an Euler system in the sense of Kolyvagin. For details, the reader is referred to \cite{Gross}, \cite{Kol}, \cite{Kol-Log}.

Now let 
\begin{equation} \label{alpha-K-eq}
\alpha_K\defeq\tr_{K_1/K}(\alpha_1)\in A_f(K)
\end{equation}
be the $K$-rational  Heegner point on $A_f$ that is considered in \cite[\S 2.3]{Kol-Log}. In light of formula \eqref{base-change-L-eq}, Assumption \ref{main-ass} and condition (b) above imply that $r_\an(f/K)=1$: by the Gross--Zagier formula (\cite[Theorem 6.3]{GZ}), this is tantamount to $\alpha_K$ being non-torsion. 

Before proceeding further, we need two auxiliary results.

\begin{lemma} \label{hom-lemma}
Let $R$ be an integral domain and let $M$ be a finitely generated $R$-module. An element $m\in M$ is torsion over $R$ if and only if $h(m)=0$ for all $h\in\Hom_R(M,R)$.
\end{lemma}

\begin{proof} It is trivial that if $m$ is torsion, then $h(m)=0$ for all $h\in\Hom_R(M,R)$ (the fact that $M$ is finitely generated plays no role here). Conversely, suppose that $m$ is not torsion. Let us fix a free $R$-submodule $M'$ of $M$ such that $m$ is an element of a basis of $M'$ over $R$ and $\rank_RM'=\rank_RM$. Then $M/M'$ is a finitely generated torsion $R$-module, hence there is $r\in R\smallsetminus\{0\}$ such that $r$ annihilates $M/M'$; this means that $rm'\in M'$ for all $m'\in M$. Let $[r]\in\Hom_R(M,M')$ be the multiplication-by-$r$ map, let $\tilde h\in\Hom_R(M',R)$ be any map such that $\tilde h(m)\not=0$ and set $h\defeq\tilde h\circ[r]\in\Hom_R(M,R)$. Then $h(m)\not=0$, and we are done. \end{proof}

All endomorphisms of $A_f$ are defined over $K$ (in fact, over $\Q$), so the Mordell--Weil group $A_f(K)$ is a module over $\cO_f$.

\begin{lemma} \label{nontorsion-lemma}
The point $\alpha_K\in A_f(K)$ is non-torsion over $\cO_f$.
\end{lemma}

\begin{proof} Since $\Q_f=\mathrm{frac}(\cO_f)$, we can equivalently show that (the natural image of) $\alpha_K$ is non-zero in $A_f(K)\otimes_{\cO_f}\!\Q_f$. There is an identification $\Q_f=(\Z\smallsetminus\{0\})^{-1}\cdot\cO_f$ of $\cO_f$-modules, so $\alpha_K$ is non-zero in $A_f(K)\otimes_{\cO_f}\!\Q_f$ because it is non-torsion over $\Z$. \end{proof}

Now we can prove 

\begin{proposition} \label{finite-number-prop}
For all but finitely many prime ideals $\p$ of $\cO_f$ the point $\alpha_K$ does not belong to the $\cO_f$-submodule $\p A_f(K)$.
\end{proposition}

\begin{proof} The group $A_f(K)$ is finitely generated over $\Z$, hence \emph{a fortiori} over $\cO_f$. Let $I$ denote the image of the homomorphism of $\cO_f$-modules
\[ \Hom_{\cO_f}\bigl(A_f(K),\cO_f\bigr)\longrightarrow\cO_f,\quad h\longmapsto h(\alpha_K). \]
By Lemma \ref{nontorsion-lemma}, $\alpha_K$ is not $\cO_f$-torsion, so Lemma \ref{hom-lemma} ensures that the ideal $I$ of $\cO_f$ is not trivial. If $\p$ is a prime ideal of $\cO_f$ such that $\alpha_K\in\p A_f(K)$, then $\p$ contains $I$. Since $\cO_f$ is an oder in the Dedekind domain $\cO_{\Q_f}$ and $I\not=(0)$, this can happen only for finitely many $\p$. \end{proof}

By Proposition \ref{finite-number-prop}, there are only finitely many prime (in fact, maximal) ideals $\p_1,\dots,\p_n$ of $\cO_f$ such that $\alpha_K\in\p_iA_f(K)$ for $i\in\{1,\dots,n\}$. Denote by $p_i$ the residue characteristic of $\p_i$, write $\mathscr P$ for the set of all prime numbers and define
\[ \mathscr S_f\defeq\mathscr P\smallsetminus\{p_1,\dots,p_n\}. \]
As in \S \ref{reduction-subsec}, for every prime $\lambda$ of $\Q_f$ let $\bar\rho_{f,\lambda}$ be the representation over the residue field of $\Q_f$ at $\lambda$ associated with $f$. By \cite[Theorem 2.1, (a)]{ribet2}, $\bar\rho_{f,\lambda}$ is irreducible for all but finitely many $\lambda$. Write $\{\ell_1,\dots,\ell_t\}$ for the set of the residue characteristics of those $\lambda$ for which $\bar\rho_{f,\lambda}$ is reducible and set
\[ \mathscr I_f\defeq\mathscr P\smallsetminus\{\ell_1,\dots,\ell_t\}. \]
Notation being as in \S \ref{choice-subsec}, now define
\[ \Xi_f\defeq\Sigma_f\cap\mathscr S_f\cap\mathscr I_f. \]

\begin{proposition} \label{density-prop}
The set $\Xi_f$ has density $1$; in particular, $\Xi_f$ is infinite.
\end{proposition}

\begin{proof} By a result of Serre, the set of prime numbers that are ordinary for $f$ has density $1$ (see, \emph{e.g.}, \cite[Proposition 2.2]{Fischman}), therefore $\Sigma_f$ has density $1$. Since $\Sigma_f\smallsetminus\Xi_f$ is finite, $\Xi_f$ has density $1$ as well. \end{proof}

\begin{remark} \label{sigma-rem}
By definition, the set $\Xi_f$ consists of all but finitely many prime numbers that are ordinary for $f$ (\emph{cf.} \S \ref{choice-subsec}).
\end{remark}

\begin{remark}
Fix an odd prime $\ell$ and an integer $M\geq1$ with $\ell\nmid M$. It is a theorem of Jochnowitz (\cite{jochnowitz-TAMS}) that there are only finitely many mod $\ell$ modular
representations of level $\Gamma_1(M)$ (see, \emph{e.g.}, \cite[Proposition 5.1.1]{ColMaz} for the case of level $\Gamma_1(M\ell^r)$ for some $r\in\N$). 
\end{remark}

Pick $p\in\Xi_f$. By \cite[Lemma 2.3]{Fischman}, the prime $p$ is unramified in $\Q_f$, hence there is a prime $\p$ of $\Q_f$ above $p$, which we fix once and for all, such that $f$ is $\p$-ordinary in the sense that $a_p(f)\in\cO_{\Q_f,\p}^\times$ (indeed, if $a_p(f)$ were divisible by all the primes of $\Q_f$ above $p$, then $a_p(f)$ would be divisible by $p$, contradicting the fact that $p$ is an ordinary prime for $f$). Now recall that $\bar\rho_{f,\p}^{ss}$ is \emph{$p$-distinguished} if its restriction to the decomposition group $G_{\Q_p}\defeq\Gal(\bar\Q_p/\Q_p)$ at $p$ can be put in the shape ${\bar\rho^{ss}_{f_\wp}|}_{G_{\Q_p}}=\bigl(\begin{smallmatrix}\varepsilon_1&*\\0&\varepsilon_2\end{smallmatrix}\bigr)$ for characters $\varepsilon_1\not=\varepsilon_2$ (see, \emph{e.g.}, \cite[\S 2]{ghate}).

\begin{lemma} \label{distinguished-lemma}
The representation $\bar\rho_{f,\p}$ is absolutely irreducible and $p$-distinguished.
\end{lemma}

\begin{proof} In light of Remark \ref{absolutely-rem}, $\bar\rho_{f,\p}$ is absolutely irreducible by our choice of $p$. Furthermore, since $f$ is $\p$-ordinary, a result of Mazur--Wiles and of Wiles (\cite[Theorem 2.1.4]{Wiles-Ordinary}) ensures that the restriction of $\rho_{f,\p}$ to $G_{\Q_p}$ has the form
\begin{equation} \label{local-eq}
{\rho_{f,\p}|}_{G_{\Q_p}}\simeq\begin{pmatrix}\delta_1&\nu\\0&\delta_2\end{pmatrix}, 
\end{equation}
where $\delta_i:G_{\Q_p}\rightarrow\cO_{\Q_f,\p}^\times$ for $i=1,2$ are characters with $\delta_2$ unramified and $\nu:G_{\Q_p}\rightarrow\cO_{\Q_f,\p}$ is a continuous map. More explicitly, $\delta_1=\xi\cdot\varepsilon_{\cyc}$ where $\xi$ is unramified and $\varepsilon_{\cyc}$ is, as in \S \ref{reduction-subsec}, the $p$-adic cyclotomic character (see, \emph{e.g.}, \cite[\S 2.1]{GV}). If, as usual, we denote reduction modulo $\p$ with a bar, then \eqref{local-eq} implies that $\bar\rho_{f,\p}$ is of the form $\Big(\begin{smallmatrix}\bar\delta_1&\ast\\0&\bar\delta_2\end{smallmatrix}\Big)$. Now let $I_p\subset G_{\Q_p}$ be the inertia subgroup. On the one hand, ${\bar\delta_2|}_{I_p}$ is trivial, as $\delta_2$ is unramified. On the other hand, $\xi$ is unramified, hence ${\bar\delta_1|}_{I_p}={\bar\varepsilon_{\cyc}|}_{I_p}$, which is non-trivial because $\bar\varepsilon_{\cyc}$ is ramified. It follows that $\bar\delta_1\not=\bar\delta_2$, so $\bar\rho_{f,\p}$ is $p$-distinguished. See, \emph{e.g.}, \cite[\S 2.3]{LV-NYJM} for more details. \end{proof}

As before, with conventions as in \S \ref{cohomological-subsec}, let $\Ta_\p(A_f)$ be the $\p$-adic Tate module of $A_f$; furthermore, write $A_f[\p]$ for the $\p$-torsion $\cO_f$-submodule of $A_f(\bar\Q)$ and for any number field $L$ let 
\begin{equation} \label{pi-f-L-eq}
\pi_{f,L}:H^1\bigl(L,\Ta_\p(A_f)\bigr)\longrightarrow H^1(L,A_f[\p]) 
\end{equation}
be the map induced by the surjection $\Ta_\p(A_f)\twoheadrightarrow A_f[\p]$. There are Kummer maps
\begin{equation} \label{pi-f-L-eq2}
\delta_{f,L}:A_f(L)\longrightarrow H^1\bigl(L,\Ta_\p(A_f)\bigr),\quad\pi_{f,L}\circ\delta_{f,L}: A_f(L)\longrightarrow H^1(L,A_f[\p]) 
\end{equation}
in Galois cohomology (see, \emph{e.g.}, \cite[Appendix A.1]{GP}). In turn, the second map gives an injection
\begin{equation} \label{delta-bar-eq}
\bar\delta_{f,L}:A_f(L)\big/\p A_f(L)\longmono H^1(L,A_f[\p]). 
\end{equation}
Thanks to our choice of $p$, the image $\bar\alpha_K$ of $\alpha_K$ in $A_f(K)/\p A_f(K)$ is non-zero, hence 
\begin{equation} \label{nonzero-eq}
\bar\delta_{f,K}(\bar\alpha_K)\not=0. 
\end{equation}
From now on let us choose as $\bar\Q_p$ a fixed algebraic closure $\bar\Q_{f,\p}$ of $\Q_{f,\p}$ and denote by $\bar\Q$ the algebraic closure of $\Q$ inside $\bar\Q_p$. Notation being as in \S \ref{notation-subsec}, the map $\iota_p$ is nothing but the tautological inclusion of $\bar\Q$ into $\bar\Q_p$, which gives a prime ideal $\fP$ of $\bar\Z$ such that $\fP\cap\cO_f=\p$; the ideal $\fP$ should be viewed as a distinguished incarnation of the prime ideal of residue characteristic $p$ from \S \ref{notation-subsec}. By construction, $f$ is $\fP$-ordinary in the sense of \S \ref{stabilization-subsec}. As in \S \ref{hida-subsec}, let $\hf\in\mathcal R[\![q]\!]$ be the $p$-adic Hida family passing through $f$. With notation as in \S \ref{hida-subsec}, we take $\cO\defeq\cO_{\Q_f,\p}$, hence $\Q_{f,\p}$ is a subfield of $\mathcal F_\wp$ for every arithmetic prime $\wp$ of $\mathcal R$. For each $\wp$, we can also choose the embedding $\iota_\wp:\mathcal F_\wp\hookrightarrow\bar\Q_p$ to be $\Q_{f,\p}$-linear. As explained in \S \ref{root-numbers-subsec}, we view $f_\wp$ and $f_\wp^\flat$ as classical modular forms with complex Fourier coefficients. In line with the notation introduced in \S \ref{cohomological-subsec}, we write $\Q_{f_\wp^\flat}$ for the Hecke field of $f_\wp^\flat$ and $\cO_{f_\wp^\flat}$ for its ring of integers. As in \S \ref{notation-subsec}, we fix an embedding $\iota_{\Q_{f_\wp^\flat}}:\Q_{f_\wp^\flat}\hookrightarrow\bar\Q$, which induces a prime ideal $\fP_{f_\wp^\flat}$ of $\cO_{f_\wp^\flat}$. Finally, as a convenient shorthand, in the following we denote $\Q_{f_\wp^\flat,\fP_{f_\wp^\flat}}$ by $\Q_{f_\wp^\flat,\fP}$ and $\cO_{f_\wp^\flat,\fP_{f_\wp^\flat}}$ by $\cO_{f_\wp^\flat,\fP}$. 

In the following statement, $\wp$ is an arithmetic prime of $\mathcal R$.

\begin{proposition} \label{trivial-prop}
If $\wp$ has trivial character and weight $k_\wp\equiv2\pmod{p-1}$, then $f_\wp$ has trivial character.
\end{proposition}

\begin{proof} Let $\omega:\F_p^\times\rightarrow\Bmu_{p-1}$ be the Teichm\"uller character, where $\Bmu_{p-1}\subset\bar\Q_p^\times$ is the group of $(p-1)$-st roots of unity. The character of $f_\wp$ is $\psi_\wp\omega^{2-k_\wp}$, and the proposition follows. \end{proof}

The main results of this paper will concern arithmetic primes $\wp$ of trivial character and weight $k_\wp\equiv2\pmod{2(p-1)}$: from here on, we will use Proposition \ref{trivial-prop} without further notice.

\subsection{Critical twist and residual representations} \label{residual-subsec}

As in \S \ref{hida-subsec}, let $\T$ be the big Galois representation associated with the Hida family $\hf$ introduced at the end of \S \ref{heegner-subsec}. Rather than in $\T$ itself, we will be interested in a suitable twist $\T^\dagger$ of $\T$ called the \emph{critical twist}, whose definition can be found, \emph{e.g.}, in \cite[Definition 2.1.3]{Howard-Inv}. With notation and terminology as in \cite{Howard-Inv}, in our case one has $k=2$ and $j=0$, hence only the ``wild'' part of the critical character plays a role. Recall from \S \ref{hida-subsec} that if $\wp$ is an arithmetic prime of $\mathcal R$, then the field $\mathcal F_\wp=\mathcal R_\wp/\wp\mathcal R_\wp=\mathrm{frac}(\mathcal R/\wp)$ is a finite extension of $\Q_p$ that we embedded into $\bar\Q_p$. Moreover, the integral domain $\mathcal R/\wp$ is the valuation ring of $\mathcal F_\wp$. The twist $\T^\dagger$ has the property that for every arithmetic prime $\wp$ of weight $k_\wp\equiv2\pmod{2(p-1)}$ and trivial character the representation $\T^\dagger_\wp\big/\wp\T^\dagger_\wp=\T^\dagger\otimes_{\mathcal R}\mathcal F_\wp$ of $G_\Q$, where $\T^\dagger_\wp$ is the localization of $\T^\dagger$ at $\wp$, is equivalent to $V^\dagger_{f_\wp}$ after a finite base change (see, \emph{e.g.}, \cite[(3.2.4)]{NP}). As a consequence, there are specialization maps
\begin{equation} \label{specialization-eq}
\T^\dagger\longepi\T^\dagger/\wp\T^\dagger\longrightarrow\T^\dagger_\wp\big/\wp\T^\dagger_\wp\simeq V^\dagger_{f_\wp}, 
\end{equation}
which in turn induce specialization maps in cohomology. Notice that, in particular, $\T^\dagger\simeq\T$ as $\mathcal R$-modules. Summing up, $\T$ and $\T^\dagger$ enjoy the following interpolation properties, up to a finite base change:
\begin{itemize}
\item the specialization of $\T$ at an arithmetic prime $\wp$ is equivalent to $V_{f_\wp}^*$;
\item if $k_\wp\equiv2\pmod{2(p-1)}$, then the specialization of $\T^\dagger$ at $\wp$ is equivalent to $V_{f_\wp}^\dagger$.
\end{itemize}
Denote by $\m_{\mathcal R}$ the maximal ideal of $\mathcal R$ and let $\F_{\mathcal R}\defeq\mathcal R/\m_{\mathcal R}$ be the residue field of $\mathcal R$. Define
\[ \bar\T\defeq\T/\mathfrak m_\mathcal R\T=\T\otimes_{\mathcal R}\F_{\mathcal R}, \]
which is a two-dimensional representation of $G_\Q$ over $\F_{\mathcal R}$. As above, let $\wp$ be an arithmetic prime of weight $k_\wp\equiv2\pmod{2(p-1)}$ and trivial character; recall from \S \ref{reduction-subsec} the reduced representations
\[ \bar\rho_{f_\wp}:G_\Q\longrightarrow\GL(\bar T_{f_\wp}),\quad\bar\rho^\dagger_{f_\wp}:G_\Q\longrightarrow\GL\bigl(\bar T^\dagger_{f_\wp}\bigr) \]
and their semisimplifications
\[ \bar\rho^{\,ss}_{f_\wp}:G_\Q\longrightarrow\GL(\bar T^{\,ss}_{f_\wp}),\quad\bar\rho^{\dagger,ss}_{f_\wp}:G_\Q\longrightarrow\GL\bigl(\bar T^{\dagger,ss}_{f_\wp}\bigr). \]
It turns out that if $\wp$ and $\wp'$ are two arithmetic primes, then 
\begin{equation} \label{semi-simple-eq}
\bar\rho^{\,ss}_{f_\wp}\simeq\bar\rho^{\,ss}_{f_{\wp'}} 
\end{equation}
after a finite base change (\cite[p. 251]{hida86a}). If $\bar\rho_{f_\wp}$ (equivalently, $\bar\rho^{\,ss}_{f_\wp}$) is irreducible and $p$-distinguished for one (hence for every) arithmetic prime $\wp$, then 
\begin{itemize}
\item $\T$ is free of rank $2$ over $\mathcal R$ (\cite[Th\'eor\`eme 7]{mazur-tilouine});
\item $\bar\T\simeq\bar\rho_{f_\wp}$ after a finite base change for all such $\wp$ (see, \emph{e.g.}, \cite[Proposition 5.4]{LV-MM}).
\end{itemize}

\begin{remark}
The condition that $\bar\rho^{ss}_{f_\wp}$ be $p$-distinguished is apparently not imposed in \cite{mazur-tilouine}, but it turns out to be necessary for the two results above to hold (\emph{cf.} \cite[p. 379]{Fouquet}).
\end{remark}

Notice that, in light of \eqref{dagger-twist-eq}, property \eqref{semi-simple-eq} no longer holds unconditionally once $\bar\rho^{\,ss}_{f_\wp}$ and $\bar\rho^{\,ss}_{f_{\wp'}}$ have been replaced by $\bar\rho^{\dagger,ss}_{f_\wp}$ and $\bar\rho^{\dagger,ss}_{f_{\wp'}}$, respectively. However, by Remark \ref{twist-rem}, if $\wp$ and $\wp'$ are two arithmetic primes such that $k_\wp\equiv k_{\wp'}\equiv2\pmod{2(p-1)}$, then
\[ \bar\rho^{\dagger,ss}_{f_\wp}\simeq\bar\rho^{\dagger,ss}_{f_{\wp'}} \] 
after a finite base change.

Finally, set $\bar\T^\dagger\defeq\T^\dagger/\m_\mathcal R\T^\dagger$. An easy computation shows that the critical character in \cite[Definition 2.1.3]{Howard-Inv} is trivial modulo $\m_\mathcal R$, hence there is a canonical identification 
\begin{equation} \label{T-T-dagger-eq}
\bar\T=\bar\T^\dagger
\end{equation}
of representations of $G_\Q$ over $\F_\mathcal R$. We write $\pi_\mathcal R:\T^\dagger\twoheadrightarrow\bar\T$ for the surjection determined by \eqref{T-T-dagger-eq} and
\begin{equation} \label{pi-R-L-eq}
\pi_{\mathcal R,L}:H^1\bigl(L,\T^\dagger\bigr)\longrightarrow H^1(L,\bar\T)
\end{equation} 
for the map in cohomology induced functorially by $\pi_{\mathcal R}$, where $L$ is a given number field.

\subsection{On the image of $\Lambda$-adic Galois representations} \label{lambda-subsec}

Let $f$, $p$ and $\hf$ be as in \S \ref{choice-subsec} and \S \ref{heegner-subsec}. Our aim is to show that (the self-dual twists of) the $\fP$-adic representations attached to suitable specializations of $\hf$ have non-solvable images, and that the same is true of their reductions. As will be apparent, a crucial role in our arguments is played by results of Fischman on the image of $\Lambda$-adic Galois representations (\cite{Fischman}).

By \cite[Theorem 3.1]{Fischman}, $\mathcal R$ is a power series ring in one variable; more precisely, there is a finite extension $\mathscr O$ of $\Z_p$ such that $\mathcal R\simeq\mathscr O[\![X]\!]$. With notation as in \eqref{hida-rep-eq}, fix a basis of $\T$ over $\mathcal R$ and let
\begin{equation} \label{rho-infty-eq}
\rho_{\hf}:G_\Q\longrightarrow\GL_2(\mathcal R)\simeq\GL_2(\mathscr O[\![X]\!]) 
\end{equation}
be Hida's big Galois representation attached to $\hf$. Since we are interested in the size of the image of $\rho_{\hf}$, we can certainly view $\rho_{\hf}$ as taking values in $\GL_2(\mathscr O[\![X]\!])$.

\begin{theorem}[Fischman] \label{fischman-thm}
The image of $\rho_{\hf}$ contains $\SL_2(\mathscr O[\![X]\!])$.
\end{theorem}

\begin{proof} We freely use notation from \cite{Fischman}, bearing in mind that the counterpart of $\mathcal R$ is denoted by $\boldsymbol{\mathrm I}$ in \cite{Fischman}. Since, by assumption, $N$ is square-free, the form $f$ has no inner twists (see, \emph{e.g.}, \cite[(3.9)]{ribet-twists}), that is, $\Gamma_f$ is trivial. By \cite[Proposition 3.12]{Fischman}, $\Gamma_F$ is also trivial, which implies that $R_F=\boldsymbol{\mathrm I}^{\Gamma_F}=\mathcal R\simeq\mathscr O[\![X]\!]$. The theorem follows immediately from \cite[Theorem 4.8]{Fischman}. \end{proof}

Let $\wp$ be an arithmetic prime of $\mathcal R$ of weight $k_\wp\equiv2\pmod{2(p-1)}$ and trivial character. Let $\wp_2$ be the arithmetic prime of $\mathcal R$ such that $f_{\wp_2}=f^\sharp$. By our choice of $f$ if $k_\wp=2$ and by Remark \ref{old-rem} if $k_\wp>2$, there is a newform $f^\flat_\wp$ of level $N$, weight $k_\wp$ and trivial character such that $(f^\flat_\wp{)}^\sharp=f_\wp$. Of course, $f_{\wp_2}^\flat=f$. Suppose that $k_\wp>2$ and, to lighten our notation, in this subsection set $g\defeq f_\wp^\flat$.

Still regarding $\rho_{\hf}$ as $\GL_2(\mathscr O[\![X]\!])$-valued, write
\begin{equation} \label{specialization-infty-eq2}
\rho_{\hf,\wp}:G_\Q\longrightarrow\GL_2(\mathscr O)\subset\GL_2(\mathscr F) 
\end{equation}
for the specialization of $\rho_{\hf}$ at $\wp$, where $\mathscr F$ is the quotient field of $\mathscr O$. This specialization map is explicitly given by replacing $X$ with $(1+p)^{k_\wp-2}-1$ in \eqref{rho-infty-eq} (\cite[Theorem II]{hida86b}). Denote by $\mathscr M$ the maximal ideal of $\mathscr O$, let $\F_{\mathscr O}\defeq\mathscr O/\mathscr M$ be the residue field of $\mathscr O$ and let
\begin{equation} \label{specialization-infty-eq}
\bar\rho_{\hf,\wp}:G_\Q\longrightarrow\GL_2(\F_{\mathscr O}) 
\end{equation}
be the reduction of $\rho_{\hf,\wp}$ modulo $\mathscr M$.

It is convenient to give 

\begin{definition}
A group representation over a commutative ring is \emph{of solvable} (respectively, \emph{non-solvable}) \emph{type} if its image is a solvable (respectively, non-solvable) group. 
\end{definition}

We want to show that the representations $\rho_g^\dagger=\rho_{g,\fP}^\dagger$ and $\bar\rho_g^\dagger$ are of non-solvable type. First of all, we prove that $\rho_{\hf,\wp}$ and $\bar\rho_{\hf,\wp}$ as in \eqref{specialization-infty-eq2} and \eqref{specialization-infty-eq} have this property. 

\begin{lemma} \label{sl-lemma}
The groups $\SL_2(\mathscr O)$ and $\SL_2(\F_{\mathscr O})$ are not solvable.
\end{lemma}            

\begin{proof} Since $p\geq5$, the non-abelian group $\PSL_2(\F_p)$ is simple (\cite[Chapter XIII, Theorem 8.4]{lang-algebra}), hence non-solvable. It follows that $\SL_2(\F_p)$, of which $\PSL_2(\F_p)$ is a quotient, is not solvable. On the other hand, $\SL_2(\F_p)$ is a subgroup of $\SL_2(\F_{\mathscr O})$, which implies that $\SL_2(\F_{\mathscr O})$ is not solvable. Finally, the natural map from $\SL_2(\Z)$ to $\PSL_2(\F_p)$ is surjective (\cite[Lemma 1.38]{shimura}), therefore $\SL_2(\Z)$ is not solvable. Since $\SL_2(\Z)$ is a subgroup of $\SL_2(\mathscr O)$, we conclude that $\SL_2(\mathscr O)$ is not solvable as well. \end{proof}

\begin{proposition} \label{non-solvable-rho-prop}
The representations $\rho_{\hf,\wp}$ and $\bar\rho_{\hf,\wp}$ are of non-solvable type.
\end{proposition}

\begin{proof} Immediate from Theorem \ref{fischman-thm} and Lemma \ref{sl-lemma}, as the image of $\rho_{\hf,\wp}$ (respectively, $\bar\rho_{\hf,\wp}$) contains the non-solvable group $\SL_2(\mathscr O)$ (respectively, $\SL_2(\F_{\mathscr O})$) and hence cannot be solvable. \end{proof}

We need three more elementary lemmas in representation theory.

\begin{lemma} \label{representations-lemma}
Let $\sigma:G\rightarrow\GL(V)$ and $\tau:G\rightarrow\GL(W)$ be representations of a group $G$. If $\sigma$ and $\tau$ are equivalent, then the groups $\sigma(G)$ and $\tau(G)$ are isomorphic.
\end{lemma}

\begin{proof} Let $h:V\xrightarrow{\simeq}W$ be an isomorphism between $\sigma$ and $\tau$. There is a homomorphism of groups
\[ \h:\sigma(G)\longrightarrow\tau(G),\quad x\longmapsto h\circ x\circ h^{-1} \]
that is well defined thanks to the $G$-equivariance of $h$. Since $h$ gives a group isomorphism $\GL(V)\xrightarrow\simeq\GL(W)$, the map $\h$ is injective. On the other hand, if $g\in G$, then $\h\bigl(\sigma(g)\bigr)=\tau(g)$, hence $\h$ is surjective as well. \end{proof}

In the following statement, by ``ring'' we mean ``commutative ring with unity''.

\begin{lemma} \label{base-change-lemma}
Let $\sigma:G\rightarrow\GL(V)$ be a representation of a group $G$ over a ring $R$, let $S$ be a ring such that $R\subset S$ and let $\sigma_S:G\rightarrow\GL(V\otimes_RS)$ be the base change of $\sigma$ to $S$. The groups $\sigma(G)$ and $\sigma_S(G)$ are (canonically) isomorphic.
\end{lemma}

\begin{proof} If $\iota_S:\GL(V)\hookrightarrow\GL(V\otimes_RS)$ is the canonical injection, then $\sigma_S=\iota_S\circ\sigma$, and the lemma follows. \end{proof}

\begin{lemma} \label{dual-isom-lemma}
Let $\sigma:G\rightarrow\GL(V)$ be a finite-dimensional representation of a group $G$ over a field $F$ and let $\sigma^*:G\rightarrow\GL(V^*)$ be its dual representation. The groups $\sigma(G)$ and $\sigma^*(G)$ are isomorphic.
\end{lemma}

\begin{proof} Set $n\defeq\dim_F(V)$. It is convenient to fix a basis of $V$ over $F$ and equip the $F$-vector space $V^*$ with the dual basis, so that we can view $\sigma$ and $\sigma^*$ as taking values in $\GL_n(F)$. By definition of $\sigma^*$, one has
\[ \sigma^*(G)=\bigl\{\sigma(g^{-1})^t\mid g\in G\bigr\}, \]
where the symbol $(\cdot)^t$ denotes transpose. It is straightforward to check that the map
\[ \sigma(G)\longrightarrow\sigma^*(G),\quad\sigma(g)\longmapsto\sigma(g^{-1})^t \]
is a group isomorphism, and the lemma is proved. \end{proof}

We immediately obtain

\begin{proposition} \label{bar-T-image-prop}
The representations $\rho_g$ and $\bar\rho_g$ are of non-solvable type.
\end{proposition}

\begin{proof} As already remarked in \S \ref{residual-subsec}, the dual representations $\rho_g^*$ and $\bar\rho_g^*$ are equivalent after finite base changes to $\rho_{\hf,\wp}$ and $\bar\rho_{\hf,\wp}$, respectively, so combining Proposition \ref{non-solvable-rho-prop} with Lemmas \ref{representations-lemma} and \ref{base-change-lemma} shows that $\rho_g^*$ and $\bar\rho_g^*$ are of non-solvable type. Now the proposition follows from Lemma \ref{dual-isom-lemma}. \end{proof}



The next lemma will allow us to deduce the non-solvability of the image of $\bar\rho_g^\dagger$ from the corresponding result for $\bar\rho_g$.

\begin{lemma} \label{solvable-lemma}
Let $\sigma:G\rightarrow\GL(V)$ and $\tau:G\rightarrow\GL(W)$ be representations of a group $G$ over a field $F$. If $\sigma$ and $\tau$ are of solvable type, then $\sigma\otimes\tau:G\rightarrow\GL(V\otimes_FW)$ is of solvable type.
\end{lemma}

\begin{proof} Let $x\in\GL(V)$ and $y\in\GL(W)$; the map
\[ V\times W\longrightarrow V\otimes_FW,\quad(v,w)\longmapsto x(v)\otimes y(w) \]
is $F$-bilinear, hence it induces an $F$-linear map
\[ \Phi(x,y):V\otimes_FW\longrightarrow V\otimes_FW \]
with inverse $\Phi(x^{-1},y^{-1})$. It follows that $\Phi(x,y)\in\GL(V\otimes_FW)$. It is immediate to check that the resulting map
\[ \Phi:\GL(V)\times\GL(W)\longrightarrow\GL(V\otimes_FW),\quad(x,y)\longmapsto\Phi(x,y)  \]
is a group homomorphism. Since the representations $\sigma$ and $\tau$ are of solvable type and the direct product of two solvable groups is solvable, the image of the group homomorphism
\[ \Phi\circ(\sigma,\tau):G\times G\longrightarrow\GL(V\otimes_FW) \]
is (isomorphic to) a quotient of a solvable group, hence it is solvable. Let $\iota_G:G\rightarrow G\times G$ be the diagonal embedding sending $g$ to $(g,g)$. On the one hand, the image of the group homomorphism $\Phi\circ(\sigma,\tau)\circ\iota_G$ is a subgroup of the image of $\Phi\circ(\sigma,\tau)$, so it is solvable. On the other hand, the image of $\Phi\circ(\sigma,\tau)\circ\iota_G$ is equal, by construction, to the image of $\sigma\otimes\tau$, and the lemma is proved. \end{proof}

Let $\pi_g$ be a uniformizer for $\cO_{\Q_g,\fP}$ and let $\F_g\defeq\cO_{\Q_g,\fP}\big/\pi_g\cO_{\Q_g,\fP}$ be the residue field of $\Q_{g,\fP}$. We can finally prove

\begin{proposition} \label{non-solvable-prop} 
The representations $\rho_g^\dagger$ and $\bar\rho_g^\dagger$ are of non-solvable type.
\end{proposition}

\begin{proof} Arguing by contradiction, suppose that $V^\dagger_g=V_g\otimes_{\Q_{g,\fP}}\Q_{g,\fP}(k/2)$ is of solvable type. The representation $\Q_{g,\fP}(-k/2)$ of $G_\Q$ is one-dimensional, hence of solvable type, therefore Lemma \ref{solvable-lemma} ensures that 
\[ V_g=V_g^\dagger\otimes_{\Q_{g,\fP}}\Q_{g,\fP}(-k/2) \]
is a representation of solvable type. This contradicts Proposition \ref{bar-T-image-prop} and shows that $\rho_g^\dagger$ is of non-solvable type. Using the reduction of $\mathscr O(-k/2)$, which is one-dimensional, one can prove the claim for $\bar\rho_g^\dagger$ in a completely analogous way. \end{proof}

In the following, let $L$ be a number field. Let us define $W^\dagger_g\defeq V^\dagger_g\big/T^\dagger_g$.

\begin{corollary} \label{solvable-coro}
If $L/\Q$ is solvable, then $H^0\bigl(L,V^\dagger_g\bigr)=H^0\bigl(L,W^\dagger_g\bigr)=H^0\bigl(L,\bar T^\dagger_g\bigr)=0$.
\end{corollary}

\begin{proof} By Remark \ref{irreducible-rem} and Proposition \ref{non-solvable-prop}, $V_g^\dagger$ is irreducible and of non-solvable type, and then $H^0\bigl(L,V^\dagger_g\bigr)=0$ by the arguments in the proof of \cite[Lemma 3.10, (1)]{LV}. On the other hand, to show that $H^0\bigl(L,W^\dagger_g\bigr)=0$ one can proceed as in the proof of \cite[Lemma 2.4, (1)]{LV-kyoto}. Finally, Remark \ref{irreducible-rem}, Lemma \ref{distinguished-lemma} and Proposition \ref{non-solvable-prop} ensure that $\bar T^\dagger_g$ is irrreducible and of non-solvable type, and then $H^0\bigl(L,\bar T^\dagger_g\bigr)=0$ as in the proof of \cite[Lemma 3.10, (2)]{LV}. \end{proof}

\begin{corollary} \label{no-torsion-coro}
If $L/\Q$ is solvable, then the $\cO_{\Q_g,\fP}$-module $H^1\bigl(L,T^\dagger_g\bigr)$ is torsion-free.
\end{corollary}

\begin{proof} Set $T\defeq T^\dagger_g$ and $\bar T\defeq\bar T^\dagger_g$. Clearly, it suffices to check that the $\pi_g$-torsion of $H^1(L,T)$ is trivial. Since $T$ is free (hence torsion-free) over $\cO_{\Q_g,\fP}$, there is a short exact sequence of Galois modules
\[ 0\longrightarrow T\xrightarrow{\pi_g\cdot}T\longrightarrow T/\pi_gT=\bar T\longrightarrow0, \] 
where the first non-trivial arrow is multiplication by $\pi_g$. Passing to cohomology, we see that the $\pi_g$-torsion submodule of $H^1(L,T)$ is a quotient of $H^0(L,\bar T)$, which is trivial thanks to Corollary \ref{solvable-coro}. \end{proof}

\section{Shafarevich--Tate groups in Hida families: the rank one case} \label{sha-sec}

In this section, we prove our results (Theorem \ref{sha-thm}) on Shafarevich--Tate groups and on images of $p$-adic \'etale Abel--Jacobi maps attached to a large class of higher (even) weight modular forms in the Hida family $\hf$ introduced at the end of \S \ref{heegner-subsec}, and so in a rank $1$ setting. Along the way, we prove a non-torsionness result (Theorem \ref{non-torsion-prop}) for certain $K$-rational Heegner cycles $y_{\wp,K}$ where $\wp$ is any arithmetic prime of weight $k_\wp>2$ such that $k_\wp\equiv2\pmod{2(p-1)}$ and trivial character. This result, which builds on work of Castella and of Ota on specializations of Howard's big Heegner points, will be crucial also in the proof of our theorems on Greenberg's conjecture when $r_{\min}(\hf)=1$ (Section \ref{rank-sec}) and on Shafarevich--Tate groups and $p$-adic Abel--Jacobi images in rank $0$ (Section \ref{sha-sec2}).

\subsection{Distinguished specialization maps} \label{distinguished-subsec}

Let $\wp$ be an arithmetic prime of $\mathcal R$ of weight $k_\wp\equiv2\pmod{2(p-1)}$ and trivial character. Let us keep the notation of \S \ref{lambda-subsec} in force; in particular, $f^\flat_\wp$ is a newform of level $N$, weight $k_\wp$ and trivial character such that $(f^\flat_\wp{)}^\sharp=f_\wp$ (and $f_{\wp_2}^\flat=f$). Recall that, by Remark \ref{isom-rep-rem}, the representations $V_{f_\wp}$ and $V_{f_\wp^\flat}$ are equivalent (possibly after a finite base change). It follows that the isomorphism in \eqref{specialization-eq} induces a (non-canonical) isomorphism $\T^\dagger_\wp\big/\wp\T^\dagger_\wp\simeq V^\dagger_{f_\wp^\flat}$.

Following Ota, we fix once and for all the distinguished $G_\Q$-equivariant specialization map
\begin{equation} \label{specialization-eq2}
\sspp^\wp_0:\T^\dagger\longrightarrow T^\dagger_{f_\wp^\flat}
\end{equation}
that is described in \cite[\S 2.6]{Ota-JNT}. As in \eqref{specialization-eq}, the map in \eqref{specialization-eq2} factors through the quotient projection $\T^\dagger\twoheadrightarrow\T^\dagger/\wp\T^\dagger$ and induces an isomorphism
\begin{equation} \label{specialization-V-eq}
\sspp^\wp_0:\T^\dagger_\wp\big/\wp\T^\dagger_\wp\overset\simeq\longrightarrow V^\dagger_{f^\flat_\wp} 
\end{equation}
of representations of $G_\Q$ after a finite base change. There is a canonical identification
\[ \T^\dagger_\wp\big/\wp\T^\dagger_\wp=\bigl(\T^\dagger/\wp\T^\dagger\bigr)\otimes_{\mathcal R/\wp}\mathcal F_\wp, \]
which shows that $\T^\dagger/\wp\T^\dagger$ sits as an $\mathcal R/\wp$-lattice inside $\T^\dagger_\wp\big/\wp\T^\dagger_\wp$. By Lemma \ref{distinguished-lemma},  $\bar\rho_{f,\p}$ is (absolutely) irreducible and $p$-distinguished. As a consequence of Remark \ref{isom-rep-rem} and \S \ref{residual-subsec}, these two properties are then enjoyed by $\bar\rho_{f_\wp^\flat}$ for every arithmetic prime $\wp$. Furthermore, all $\bar\rho^\dagger_{f_\wp^\flat}$ are (absolutely) irreducible as well. It follows (see \S \ref{residual-subsec}) that $\T$ is free of rank $2$ over $\mathcal R$ and $\bar\T$ is equivalent (up to a finite base change) to $\bar\rho_{f_\wp}$ for all $\wp$. 

Choose a finite extension $L/\Q_p$ inside $\bar\Q_p$ such that $\mathcal F_\wp\cup\Q_{f_\wp,\fP}\subset L$ and write $\cO_L$ for the valuation ring of $L$. Isomorphism \eqref{specialization-V-eq} induces an isomorphism
\begin{equation} \label{specialization-eq4}
\sspp_0^\wp:\bigl(\T^\dagger/\wp\T^\dagger\bigr)\otimes_{\mathcal R/\wp}\cO_L\overset\simeq\longrightarrow T^\dagger_{f_\wp^\flat}\otimes_{\cO_{f_\wp^\flat,\fP}}\cO_L
\end{equation}
(see, \emph{e.g.}, \cite[\S 7.3]{KLZ}). With notation as in \S \ref{residual-subsec} and using equality \eqref{T-T-dagger-eq}, the reduction of $\bigl(\T^\dagger/\wp\T^\dagger\bigr)\otimes_{\mathcal R/\wp}\cO_L$ modulo the maximal ideal of $\cO_L$ is a finite base change of $\bar\T$. Thus, the isomorphism in \eqref{specialization-eq4} determines, after a finite base change (\emph{i.e.}, over $\bar\F_p$), an isomorphism
\begin{equation} \label{specialization-eq3}
\bsspp^\wp_0:\bar\T\overset\simeq\longrightarrow\bar T_{f_\wp^\flat}^\dagger 
\end{equation}
of representations of $G_\Q$ that makes the square
\begin{equation} \label{specialization-square}
\xymatrix@C=35pt{\T^\dagger\ar[r]^-{\sspp^\wp_0}\ar@{->>}[d]^-{\pi_{\mathcal R}}&T^\dagger_{f_\wp^\flat}\ar@{->>}[d]\\
                   \bar\T\ar[r]^-{\bsspp^\wp_0}_-\simeq&\bar T_{f_\wp^\flat}^\dagger}
\end{equation}
commute; here $\pi_{\mathcal R}$ is, as at the end of \S \ref{hida-subsec}, given by reduction modulo $\mathfrak m_\mathcal R$ and the vertical arrow on the right is the canonical projection. By functoriality, the maps in \eqref{specialization-eq2} and \eqref{specialization-eq3} determine maps
\begin{equation} \label{specialization-eq5}
\sspp^\wp_{0,\mathcal F}:H^1(\mathcal F,\T^\dagger)\longrightarrow H^1\Bigl(\mathcal F,T^\dagger_{f_\wp^\flat}\Bigr),\quad\bsspp^\wp_{0,\mathcal F}:H^1(\mathcal F,\bar\T)\overset\simeq\longrightarrow H^1\Bigl(\mathcal F,\bar T^\dagger_{f_\wp^\flat}\Bigr)
\end{equation}
for any number field $\mathcal F$, and then \eqref{specialization-square} gives a commutative square
\begin{equation} \label{specialization-square2}
\xymatrix@R=32pt@C=37pt{H^1(\mathcal F,\T^\dagger)\ar[r]^-{\sspp^\wp_{0,\mathcal F}}\ar[d]^-{\pi_{\mathcal R,\mathcal F}}&H^1\Big(\mathcal F,T^\dagger_{f_\wp^\flat}\Big)\ar[d]^-{\pi_{\wp,\mathcal F}}\\
                   H^1(\mathcal F,\bar\T)\ar[r]^-{\bsspp^\wp_{0,\mathcal F}}_-\simeq&H^1\Big(\mathcal F,\bar T_{f_\wp^\flat}^\dagger\Big)}
\end{equation}
in Galois cohomology, where $\pi_{\mathcal R,\mathcal F}$ is the map in \eqref{pi-R-L-eq} and $\pi_{\wp,\mathcal F}$ is induced by the surjection $T^\dagger_{f_\wp^\flat}\twoheadrightarrow\bar T_{f_\wp^\flat}^\dagger$. 

\subsection{\'Etale Abel--Jacobi maps and Heegner cycles} \label{cycles-subsec}

We briefly review basic properties of Heegner cycles attached to higher (even) weight modular forms. These cycles were originally introduced by Nekov\'a\v{r} in \cite{Nek} in order to generalize Kolyvagin's theory of Euler systems of Heegner points on modular abelian varieties to Chow groups of Kuga--Sato varieties. We will not work in maximal generality here, but rather consider only Heegner cycles for modular forms in the Hida family $\hf$ that was fixed in \S \ref{heegner-subsec}.

Let $\wp$ be an arithmetic prime of $\mathcal R$ of even weight $k_\wp>2$ and trivial character. Set $g\defeq f_\wp^\flat$ and $k\defeq k_\wp$. With notation as in \cite{LV}, let $\tilde\E_N^{k-2}$ be the Kuga--Sato variety of level $N$ and weight $k$ (see, \emph{e.g.}, \cite[\S 2.2]{LV}) and for any number field $L$ let $\CH^{k/2}\bigl(\tilde\E_N^{k-2}/L\bigr)_0$ be the Chow group of rational equivalence classes of codimension $k/2$ cycles on $\tilde\E_N^{k-2}$ defined over $L$ that are homologically trivial, \emph{i.e.}, belong to the kernel of the cycle class map in $\ell$-adic \'etale cohomology (see, \emph{e.g.}, \cite[\S 3.3]{Andre} or \cite[Chapter VI, \S 9]{Milne}; see also \cite[\S 1.4]{Nek3} for details on the independence of the above-mentioned kernel of the prime number $\ell$). Recall that, as a shorthand, $\cO_{\Q_g,\fP}$ stands for $\cO_{\Q_g,\fP_g}$. As explained, \emph{e.g.}, in \cite[\S 2.3]{LV}, there is an $\cO_{\Q_g,\fP}$-linear \'etale Abel--Jacobi map (denoted by $\AJ_{g,\fP,L}$ in \cite{LV})
\begin{equation} \label{AJ-eq}
\Phi_{g,\fP,L}:{\CH}^{k/2}\bigl(\tilde\E_N^{k-2}/L\bigr)_0\otimes_\Z\cO_{\Q_g,\fP}\longrightarrow H^1\bigl(L,T_g^\dagger\bigr).
\end{equation}
We set $\Lambda_{g,\fP}(L)\defeq\im(\Phi_{g,\fP,L})$. Now let $H^1_f\bigl(L,T^\dagger_g\bigr)$ be the Selmer group of $T^\dagger_g$ over $L$ in the sense of Bloch--Kato (\cite{BK}; see also, \emph{e.g.}, \cite[\S 2.4]{LV}). By definition, $H^1_f\bigl(L,T^\dagger_g\bigr)$ is an $\cO_{\Q_g,\fP}$-submodule of $H^1\bigl(L,T^\dagger_g\bigr)$; furthermore, it is well known that $H^1_f\bigl(L,T^\dagger_g\bigr)$ is finitely generated over $\cO_{\Q_g,\fP}$. By \cite[Corollary 2.7, (2)]{LV}, there is an inclusion of $\cO_{\Q_g,\fP}$-modules $\Lambda_{g,\fP}(L)\subset H^1_f\bigl(L,T^\dagger_g\bigr)$, hence $\Lambda_{g,\fP}(L)$ is finitely generated over $\cO_{\Q_g,\fP}$ as well. 

In the following definition, $L$ is a number field.

\begin{definition} \label{algebraic-rank-def}
The \emph{algebraic $\fP$-rank of $g$ over $L$} is $r_{\alg,\fP}(g/L)\defeq\rank_{\cO_{\Q_g,\fP}}\Lambda_{g,\fP}(L)$.
\end{definition}


Upon combining Greenberg's conjecture (Conjecture \ref{greenberg-conj}) with the conjectures of Birch--Swinnerton-Dyer (\cite[\S 1]{Tate-BSD}) and of Beilinson--Bloch--Kato (\cite[Conjecture 2.10]{LV}) on $L$-functions of modular forms, it is natural to propose the following

\begin{conjecture} \label{main-conj}
The equality
\[ r_{\alg,\fP}\bigl(f_\wp^\flat/\Q\bigr)=r_{\min}\bigl(\hf\bigr) \] 
holds for all but finitely many arithmetic primes $\wp$ of even weight and trivial character.
\end{conjecture}

Later we shall prove results (Corollaries \ref{sha-coro} and \ref{sha-coro2}) in the direction of this conjecture, which can be formulated for all Hida families, not only those of square-free tame level.

Now resume the notation $g=f_\wp^\flat$. Let $D_K$ be the discriminant of $K$. Using the Heegner points $x_c\in X_0(N)(K_c)$ from \S \ref{heegner-subsec}, Nekov\'a\v{r} defined in \cite{Nek} a collection of Heegner cycles 
\[ y_{\wp,c}\in\Lambda_{g,\fP}(K_c)\subset H^1\bigl(K_c,T^\dagger_g\bigr) \]
indexed on the integers $c\geq1$ such that $(c,ND_Kp)=1$ (see \cite[\S 3.1]{LV} and \cite[\S 5]{Nek} for details). Define the $K$-rational Heegner cycle
\begin{equation} \label{y-K-eq}
y_{\wp,K}\defeq\cores_{K_1/K}(y_{\wp,1})\in\Lambda_{g,\fP}(K)\subset H^1\bigl(K,T^\dagger_g\bigr),
\end{equation}
which is a higher weight counterpart of the point $\alpha_K$ in \eqref{alpha-K-eq}.

In the next theorem, $\Sha_\fP(g/K)$ is the $\fP_g$-primary Shafarevich--Tate group of $g$ over $K$, whose definition will be recalled in \eqref{sha-def-eq} below.

\begin{theorem}[Nekov\'a\v{r}] \label{nekovar-thm}
If $y_{\wp,K}$ is non-torsion, then
\begin{enumerate}
\item $\Lambda_{g,\fP}(K)\otimes_\Z\Q=\Q_{g,\fP}\cdot y_{\wp,K}$;
\vskip 1mm
\item $\Sha_\fP(g/K)$ is finite.
\end{enumerate}
\end{theorem}

\begin{proof} This is \cite[Theorem 13.1]{Nek}.  \end{proof}

Of course, part (1) implies that $r_{\alg,\fP}(g/K)=1$.

\begin{remark} 
To be precise, Nekov\'a\v{r} proved Theorem \ref{nekovar-thm} in \cite{Nek} under the assumption that $p\nmid2N(k-2)!\varphi(N)$, which is clearly unsuitable for applications to $p$-adic analytic families of modular forms of varying weight. However, as pointed out in \cite[\S 6.5]{Nek2}, the assumption on $p$ can be relaxed by simply requiring that $p\nmid2N$. 
\end{remark}

\begin{remark}
If $y_{\wp,K}$ is non-torsion, then one can describe the structure of the finite $\cO_{\Q_g,\fP}$-module $\Sha_{g,\fP}(g/K)$ in terms of the collection ${\{y_{\wp,c}\}}_{c\geq1}$ of Heegner cycles (\cite[Theorem 7.3]{Masoero}).
\end{remark}

For completeness, we give the counterpart of Definition \ref{algebraic-rank-def} in weight $2$. Again, let $L$ be a number field.

\begin{definition} \label{algebraic-rank-def2}
The \emph{algebraic $\fP$-rank of $f$ over $L$} is $r_{\alg,\fP}(f/L)\defeq\rank_{\cO_{f,\p}}\bigl(A_f(L)\otimes_{\cO_f}\cO_{f,\p}\bigr)$.
\end{definition}

In particular, if $A_f$ is an elliptic curve (\emph{i.e.}, $a_n(f)\in\Q$ for all $n\geq1$), then $r_{\alg,\fP}(f/L)=\rank_{\Z_p}\bigl(A_f(L)\otimes_\Z\Z_p\bigr)=\rank_\Z A_f(L)$.

\subsection{Comparison of Abel--Jacobi images over $\Q$ and over $K$} \label{AJ-subsec}

Let us keep the notation of \S \ref{cycles-subsec}. We want to show that, under our running assumptions, there is an isomorphism of $\cO_{g,\fP}$-modules
\[ \Lambda_{g,\fP}(\Q)\overset\simeq\longrightarrow{\Lambda_{g,\fP}(K)}^{\Gal(K/\Q)}. \]
This map is canonically induced by restriction; later on, it will be regarded as an equality. 

We begin by showing that restriction gives rise to an injection between the Abel--Jacobi images we are interested in.

\begin{proposition} \label{lambda-prop1}
Restriction induces an injection of $\cO_{\Q_g,\fP}$-modules
\[ \Lambda_{g,\fP}(\Q)\longmono{\Lambda_{g,\fP}(K)}^{\Gal(K/\Q)}. \]
\end{proposition}

\begin{proof} Set $\cG\defeq\Gal(K/\Q)$. As explained, \emph{e.g.}, in \cite[\S A.9.17]{BG}, there is a natural $\cO_{\Q_g,\fP}$-linear base change morphism 
\begin{equation} \label{base-change-eq}
\CH^{k/2}\bigl(\tilde\E_N^{k-2}/\Q\bigr)_0\otimes_\Z\cO_{\Q_g,\fP}\longrightarrow\left(\CH^{k/2}\bigl(\tilde\E_N^{k-2}/K\bigr)_0\otimes_\Z\cO_{\Q_g,\fP}\right)^{\!\cG}. 
\end{equation}
By composing \eqref{base-change-eq} with (the restriction of) the Abel--Jacobi map $\Phi_{g,\fP,K}$ in \eqref{AJ-eq}, we get a map
\[ \psi_{\Q,K}:\CH^{k/2}\bigl(\tilde\E_N^{k-2}/\Q\bigr)_0\otimes_\Z\cO_{\Q_g,\fP}\longrightarrow H^1\bigl(K,T^\dagger_g\bigr)^\cG. \]
On the other hand, since $K/\Q$ is solvable, Corollary \ref{solvable-coro} implies that $H^0\bigl(K,T^\dagger_g\bigr)=0$. It follows that restriction induces an isomorphism
\[ \res_{K/\Q}:H^1\bigl(\Q,T^\dagger_g\bigr)\overset\simeq\longrightarrow{H^1\bigl(K,T^\dagger_g\bigr)}^\cG. \] 
Now it turns out (see, \emph{e.g.}, the proof of \cite[Proposition 6]{Charles}) that the composition
\[ \res_{K/\Q}^{-1}\circ\psi_{\Q,K}:\CH^{k/2}\bigl(\tilde\E_N^{k-2}/\Q\bigr)_0\otimes_\Z\cO_{\Q_g,\fP}\longrightarrow H^1\bigl(\Q,T^\dagger_g\bigr) \]
coincides with the Abel--Jacobi map $\Phi_{g,\fP,\Q}$. Since, by definition, $\Lambda_{g,\fP}(\Q)=\im(\Phi_{g,\fP,\Q})$ and $\Lambda_{g,\fP}(K)=\im(\Phi_{g,\fP,K})$, we obtain a natural injection of $\cO_{\Q_g,\fP}$-modules
\[ \Lambda_{g,\fP}(\Q)\longmono{\Lambda_{g,\fP}(K)}^\cG,\quad x\longmapsto\res_{K/\Q}(x), \]
as desired. \end{proof}

\begin{lemma} \label{res-onto-lemma}
The restriction
\[ \left(\CH^{k/2}\bigl(\tilde\E_N^{k-2}/K\bigr)_0\otimes_\Z\cO_{\Q_g,\fP}\right)^{\!\Gal(K/\Q)}\longrightarrow{\Lambda_{g,\fP}(K)}^{\Gal(K/\Q)} \]
of the Abel--Jacobi map $\Phi_{g,\fP,K}$ is surjective.
\end{lemma}

\begin{proof} For short, set $\cG\defeq\Gal(K/\Q)$, $\CH_K\defeq\CH^{k/2}\bigl(\tilde\E_N^{k-2}/K\bigr)_0\otimes_\Z\cO_{\Q_g,\fP}$, $\Lambda_K\defeq\Lambda_{g,\fP}(K)$ and $\Phi_K\defeq\Phi_{g,\fP,K}$. The tautological short exact sequence of $\cO_{\Q_g,\fP}[\cG]$-modules
\[ 0\longrightarrow\ker(\Phi_K)\longrightarrow\CH_K\xrightarrow{\Phi_K}\Lambda_K\longrightarrow0 \]
yields a long exact sequence
\begin{equation} \label{long-Phi-eq}
0\longrightarrow H^0\bigl(\cG,\ker(\Phi_K)\bigr)\longrightarrow H^0(\cG,\CH_K)\xrightarrow{\Phi_K}H^0(\cG,\Lambda_K)\longrightarrow H^1\bigl(\cG,\ker(\Phi_K)\bigr)\longrightarrow\dots 
\end{equation}
in Galois cohomology. Since $\#\cG=2$, the $\cO_{\Q_g,\fP}$-module $H^1\bigl(\cG,\ker(\Phi_K)\bigr)$ is $2$-torsion, hence $H^1\bigl(\cG,\ker(\Phi_K)\bigr)=0$ because $p$ is odd. It follows from \eqref{long-Phi-eq} that the restriction of $\Phi_K$ to $\CH_K^\cG$ surjects onto $\Lambda_K^\cG$, as was to be shown. \end{proof}

We are in a position to prove

\begin{proposition} \label{lambda-prop2}
Restriction induces an isomorphism of $\cO_{g,\fP}$-modules
\[ \Lambda_{g,\fP}(\Q)\overset\simeq\longrightarrow{\Lambda_{g,\fP}(K)}^{\Gal(K/\Q)}. \]
\end{proposition}

\begin{proof} Retain the notation introduced above; set also $\CH_\Q\defeq\CH^{k/2}\bigl(\tilde\E_N^{k-2}/\Q\bigr)_0\otimes_\Z\cO_{\Q_g,\fP}$, $\Lambda_\Q\defeq\Lambda_{g,\fP}(\Q)$, $\Phi_\Q\defeq\Phi_{g,\fP,\Q}$. Since $[K:\Q]=2$ and $p$ is odd, base change gives an isomorphism $\CH_\Q\simeq\CH_K^\cG$ of $\cO_{g,\fP}$-modules (\emph{cf.} \cite[\S 4.1.2]{LV-tamagawa}). On the other hand, there is a commutative square
\[
\xymatrix@C=35pt@R=30pt{\CH_\Q\ar@{->>}[r]^-{\Phi_\Q}\ar[d]^-\simeq&\Lambda_\Q\ar@{^{(}->}[d]^-{\res_{K/\Q}}\\ \CH_K^\cG\ar@{->>}[r]^-{\Phi_K}&\Lambda_K^\cG}
\]
in which the lower horizontal map is surjective by Lemma \ref{res-onto-lemma} and the injectivity of $\res_{K/\Q}$ is the content of Proposition \ref{lambda-prop1}. We conclude that $\res_{K/\Q}$ is necessarily surjective, hence an isomorphism. \end{proof}

\begin{remark} \label{lambda-inj-rem}
By Corollary \ref{no-torsion-coro}, $\Lambda_{g,\fP}(\Q)$ and $\Lambda_{g,\fP}(K)$ are torsion-free, therefore there are canonical $\cO_{\Q_g,\fP}$-linear injections $\Lambda_{g,\fP}(\Q)\hookrightarrow\Lambda_{g,\fP}(\Q)\otimes_\Z\Q$ and $\Lambda_{g,\fP}(K)\hookrightarrow\Lambda_{g,\fP}(K)\otimes_\Z\Q$. 
\end{remark}

\subsection{Big Heegner points and their specializations} \label{big-subsec}

As before, let $\hf\in\mathcal R[\![q]\!]$ be the $p$-adic Hida family through $f$ and let $K$ be the imaginary quadratic field that was chosen in \S \ref{heegner-subsec}. In his article \cite{Howard-Inv} on the variation of Heegner points, Howard introduced certain Galois cohomology classes $P_c\in H^1(K_c,\T^\dagger)$ called \emph{big Heegner points}. Here $c\geq1$ can be any positive integer coprime to $N$ and, as in \S \ref{cycles-subsec}, $K_c$ is the ring class field of $K$ of conductor $c$. The construction of the classes $P_c$, which interpolate Kummer images of Heegner points on abelian varieties of $\GL_2$ type attached to weight $2$ cusp forms, is quite intricate and, in fact, immaterial for our goals; the reader is referred to \cite[\S 2.2]{Howard-Inv}, \cite[\S 4.2]{CasHeeg} and \cite[Section 3]{Ota-JNT} for details.

Let $\wp$ be an arithmetic prime of $\mathcal R$ of weight $k_\wp\equiv2\pmod{2(p-1)}$ and trivial character; fix an integer $c\geq1$ coprime to $ND_Kp$. It is natural to consider $\sspp^\wp_{0,K_c}(P_c)\in H^1\bigl(K_c,T^\dagger_{f_\wp}\bigr)$, where the specialization map $\sspp^\wp_{0,K_c}$ is as in \eqref{specialization-eq5}, and compare it to (the Kummer image of) the point $\alpha_c$ of \S \ref{heegner-subsec} (respectively, the cycle $y_{\wp,c}$ of \S \ref{cycles-subsec}) if $k_\wp=2$ (respectively, $k_\wp>2$). The specialization results that we are interested in have been obtained by Castella (\cite{CasHeeg}) and by Ota (\cite{Ota-JNT}).

\begin{remark}
We identify (big) Heegner points and Heegner cycles with their images via the natural group homomorphisms
\[ H^1(L,M)\longrightarrow H^1(L,M\otimes\bar\Z_p),\quad H^1(L,N)\longrightarrow H^1(L,N\otimes\bar\F_p) \]
where $M\in\bigl\{\Ta_\p(A_f),T^\dagger_f,\T^\dagger,T^\dagger_{f_\wp^\flat}\bigr\}$, $N\in\bigl\{A_f[\p],\bar T^\dagger_f,\bar\T,\bar T^\dagger_{f_\wp^\flat}\bigr\}$ and $L$ is either $K$ or a ring class field of $K$.
\end{remark}

For any number field $L$, let 
\begin{equation} \label{xi-eq}
\xi_{f,L}:H^1\bigl(L,T^\dagger_f\bigr)\overset\simeq\longrightarrow H^1\bigl(L,\Ta_\p(A_f)\bigr)
\end{equation}
be the isomorphism induced, thanks to our choice of the lattice $T^\dagger_f$ in weight $2$, by \eqref{V-2-eq} via the Weil pairing. Moreover, let $\delta_{f,L}:A_f(L)\rightarrow H^1\bigl(L,\Ta_\p(A_f)\bigr)$ be the Kummer map in \eqref{pi-f-L-eq2}.

\begin{theorem}[Castella, Ota] \label{c-o-thm}
There exist $d=d(c)\in\bar\Z_p^\times$ and $e=e(c)\in\bar\Z_p^\times$ such that 
\begin{enumerate}
\item $\xi_{f,K_c}\bigl(\sspp^{\wp_2}_{0,K_c}(P_c)\bigr)=d\cdot\delta_{f,K_c}(\alpha_c)$;
\vskip 1mm
\item $\sspp^\wp_{0,K_c}(P_c)=e\cdot y_{\wp,c}$.
\end{enumerate}   
\end{theorem}

\begin{proof} Part (2) is \cite[Theorem 6.5]{CasHeeg}, and part (1) follows by the same arguments (see \cite[Remark 6.6]{CasHeeg}). Observe that, in our setting, the constant appearing in \cite[Theorem 6.5]{CasHeeg} is a $p$-adic unit. See also \cite[Theorem 1.2]{Ota-JNT} for a refinement of the main result of \cite{CasHeeg} that works under our assumptions, which are slightly weaker than those in \cite{CasHeeg}. \end{proof}

\begin{remark} \label{c-o-rem}
Castella proved Theorem \ref{c-o-thm} in \cite{CasHeeg} under the extra assumption that ${\bar\rho_{f,\p}|}_{G_K}$ is irreducible: we shall elaborate on this condition in Appendix \ref{appendix}. Note, however, that the main theorem of \cite{Ota-JNT} does not require this condition.
\end{remark}

\begin{remark}
Theorem \ref{c-o-thm} was proved by Castella and by Ota as a corollary of results (see \cite[Theorem 1.2]{Ota-JNT}) on the generalized Heegner cycles of Bertolini--Darmon--Prasanna (\cite{BDP}).
\end{remark}

\begin{remark}
The analogue of part (1) of Theorem \ref{c-o-thm} when $p$ divides the conductor of $f$, which is irrelevant to our paper, was essentially proved (albeit somewhat in disguise) by Howard in \cite[Section 3]{Howard-derivatives}. The reader is referred to \cite[\S 3.17.2]{LV-tamagawa} for details.
\end{remark}

For the reader's benefit, the following proposition collects all isomorphisms that will be needed in the proof of Theorem \ref{non-torsion-prop}.

\begin{proposition} \label{all-iso-prop}
Let $L$ be a number field and let $\wp$ be an arithmetic prime of $\mathcal R$ of weight $k_\wp\equiv2\pmod{2(p-1)}$ and trivial character. There are group isomorphisms
\begin{enumerate}
\item $\bsspp^\wp_{0,L}:H^1(L,\bar\T)\overset\simeq\longrightarrow H^1\Bigl(L,\bar T^\dagger_{f_\wp^\flat}\Bigr)$;
\item $\xi_{f,L}:H^1\bigl(L,T^\dagger_f\bigr)\overset\simeq\longrightarrow H^1\bigl(L,\Ta_\p(A_f)\bigr)$;
\item $\bar\xi_{f,L}:H^1\bigl(L,\bar T^\dagger_f\bigr)\overset\simeq\longrightarrow H^1\bigl(L,A_f[\p]\bigr)$.
\end{enumerate}
\end{proposition}

\begin{proof} The isomorphism in (1) was introduced in \eqref{specialization-eq5}, whereas the isomorphism in (2) is \eqref{xi-eq}. Finally, the isomorphism in (3) is induced by $\xi_{f,L}$ via reduction modulo $\p$. \end{proof}

Let $\wp$ be an arithmetic prime of $\mathcal R$ with trivial character and weight $k_\wp>2$ such that $k_\wp\equiv2\pmod{2(p-1)}$. Recall the point $\alpha_K$ in \eqref{alpha-K-eq} and the cycle $y_{\wp,K}$ in \eqref{y-K-eq}. The next result, for which Assumption \ref{main-ass} is in force, will play a key role in the proofs of our theorems on the arithmetic of Hida families; in what follows, we will freely use the isomorphisms gathered in Proposition \ref{all-iso-prop}.

\begin{theorem} \label{non-torsion-prop}
The Heegner cycle $y_{\wp,K}$ is non-torsion over $\cO_{f_\wp^\flat,\fP}$.
\end{theorem}

\begin{proof} Recall the prime $\p$ of $\Q_f$ above $p$ that was fixed in \S \ref{heegner-subsec}; the prime ideal $\fP$ of $\bar\Z$ was chosen so that $\fP\cap\cO_f=\p$. Corestriction is the unique map of cohomological functors $H^*(K_1,-)\rightarrow H^*(K,-)$ that in dimension $0$ coincides with the trace, hence the square  
\[ \xymatrix@R=32pt@C=65pt{A_f(K_1)\ar[d]^-{\tr_{K_1/K}}\ar[r]^-{\pi_{f,K_1}\circ\,\delta_{f,K_1}}&H^1\bigl(K_1,A_f[\p]\bigr)\ar[d]^-{\cores_{K_1/K}}\\
                   A_f(K)\ar[r]^-{\pi_{f,K}\circ\,\delta_{f,K}}&H^1\bigl(K,A_f[\p]\bigr)} \]
is commutative (here $\pi_{f,K}\circ\delta_{f,K}$ and $\pi_{f,K_1}\circ\delta_{f,K_1}$ are as in \eqref{pi-f-L-eq2}). With notation as in \eqref{delta-bar-eq}, it follows that
\begin{equation} \label{cores-delta-eq}
\cores_{K_1/K}\bigl((\pi_{f,K_1}\circ\delta_{f,K_1})(\alpha_1)\bigr)=(\pi_{f,K}\circ\delta_{f,K})(\alpha_K)=\bar\delta_{f,K}(\bar\alpha_K)\not=0, 
\end{equation}
where the nonvanishing on the right is \eqref{nonzero-eq}. As before, let $\wp_2$ be the arithmetic prime of $\mathcal R$ such that $f_{\wp_2}=f^\sharp$, so $f_{\wp_2}^\flat=f$. For any number field $L$, square \eqref{specialization-square2} gives a commutative diagram
\begin{equation} \label{commutative-eq}
\xymatrix@R=32pt@C=42pt{H^1\bigl(L,\Ta_\p(A_f)\bigr)\ar[d]^-{\pi_{f,L}}&H^1\bigl(L,T^\dagger_f\bigr)\ar[d]^-{\pi_{\wp_2,L}}\ar[l]^-\simeq_-{\xi_{f,L}}&H^1(L,\T^\dagger)\ar[r]^-{\sspp^\wp_{0,L}}\ar[l]_-{\sspp^{\wp_2}_{0,L}}\ar[d]^-{\pi_{\mathcal R,L}}&H^1\Big(L,T^\dagger_{f_\wp^\flat}\Big)\ar[d]^-{\pi_{\wp,L}}\\H^1\bigl(L,A_f[\p]\bigr)&H^1\bigl(L,\bar T^\dagger_f\bigr)\ar[l]^-\simeq_-{\bar\xi_{f,L}}&
                   H^1(L,\bar\T)\ar[l]_-{\bsspp^{\wp_2}_{0,L}}^-\simeq&H^1\Big(L,\bar T_{f_\wp^\flat}^\dagger\Big)\ar[l]_-{(\bsspp^\wp_{0,L})^{-1}}^-\simeq} 
\end{equation}
in which the middle and right horizontal maps are those appearing in \eqref{specialization-eq5} and the vertical ones are as in \eqref{pi-f-L-eq}, \eqref{pi-R-L-eq} and \S \ref{distinguished-subsec}. We remark that, although our notation does not reflect this, all characteristic $0$ (respectively, residual) representations in \eqref{commutative-eq} are base changed to $\bar\Z_p$  (respectively, $\bar\F_p$). Now recall the big Heegner point $P_1\in H^1(K_1,\T^\dagger)$ from \S \ref{big-subsec} and define
\[ P_K\defeq\cores_{K_1/K}(P_1)\in H^1(K,\T^\dagger). \]
By part (2) of Theorem \ref{c-o-thm}, $\sspp^\wp_{0,K_1}(P_1)=e\cdot y_{\wp,1}$ for some $e\in\bar\Z_p^\times$. Since specialization commutes (in the obvious sense) with corestriction, it follows that
\[ \sspp^\wp_{0,K}(P_K)=\cores_{K_1/K}\bigl(\sspp^\wp_{0,K_1}(P_1)\bigr)=e\cdot\cores_{K_1/K}(y_{\wp,1})=e\cdot y_{\wp,K}, \]
whence
\begin{equation} \label{reduction-eq1}
\pi_{\wp,K}\bigl(\sspp^\wp_{0,K}(P_K)\bigr)=\bar e\cdot\pi_{\wp,K}(y_{\wp,K})
\end{equation}
with $\bar e\in\bar\F_p^\times$. On the other hand, by part (1) of Theorem \ref{c-o-thm}, $\xi_{f,K_1}\bigl(\sspp^{\wp_2}_{0,K_1}(P_1)\bigr)=d\cdot\delta_{f,K_1}(\alpha_1)$ for some $d\in\bar\Z_p^\times$, hence \eqref{cores-delta-eq} gives
\begin{equation} \label{reduction-eq2}
\begin{split}
 \pi_{f,K}\Big(\xi_{f,K}\bigl(\sspp^{\wp_2}_{0,K}(P_K)\bigr)\Big)&=\pi_{f,K}\Big(\cores_{K_1/K}\bigl(\xi_{f,K_1}\bigl(\sspp^{\wp_2}_{0,K_1}(P_1)\bigr)\bigr)\Big)\\
 &=\bar d\cdot\pi_{f,K}\Big(\cores_{K_1/K}\bigl(\delta_{f,K_1}(\alpha_1)\bigr)\Big)\\
 &=\bar d\cdot\cores_{K_1/K}\bigl((\pi_{f,K_1}\circ\delta_{f,K_1})(\alpha_1)\bigr)\\
 &=\bar d\cdot\bar\delta_{f,K}(\bar\alpha_K)
\end{split} 
\end{equation}
with $\bar d\in\bar\F_p^\times$. For simplicity, set
\[ \psi_K\defeq\bar\xi_{f,K}\circ\bsspp^{\wp_2}_{0,K}\circ\,(\bsspp^\wp_{0,K})^{-1}:H^1\Big(K,\bar T_{f_\wp^\flat}^\dagger\Big)\overset\simeq\longrightarrow H^1\bigl(K,A_f[\p]\bigr). \]
In light of \eqref{reduction-eq1} and \eqref{reduction-eq2}, the commutativity of \eqref{commutative-eq} with $L=K$ ensures that
\[ \begin{split}
   \psi_K\bigl(\pi_{\wp,K}(y_{\wp,K})\bigr)&=\bar e^{-1}\cdot\psi_K\Big(\pi_{\wp,K}\bigl(\sspp^\wp_{0,K}(P_K)\bigr)\Big)\\
   &=\bar e^{-1}\cdot\pi_{f,K}\Big(\xi_{f,K}\bigl(\sspp^{\wp_2}_{0,K}(P_K)\bigr)\Big)\\
   &=\bar e^{-1}\bar d\cdot\bar\delta_{f,K}(\bar\alpha_K) 
   \end{split} \]
with $\bar e^{-1}\bar d\in\bar\F_p^\times$. But $\bar\delta_{f,K}(\bar\alpha_K)\not=0$ by \eqref{nonzero-eq}, hence $y_{\wp,K}\not=0$ in $H^1\bigl(K,T^\dagger_{f_\wp^\flat}\otimes\bar\Z_p\bigr)$. This immediately implies that $y_{\wp,K}\not=0$ in $H^1\bigl(K,T^\dagger_{f_\wp^\flat}\bigr)$. Finally, by Corollary \ref{no-torsion-coro}, the $\cO_{f_\wp^\flat,\fP}$-module $H^1\bigl(K,T^\dagger_{f_\wp^\flat}\bigr)$ is torsion-free, hence $y_{\wp,K}$ is not torsion. \end{proof}

\subsection{Results on Shafarevich--Tate groups} \label{sha-subsec}

Suppose that the abelian variety $A_f$ is an elliptic curve; this is equivalent to requiring that $a_n(f)\in\Q$ for all $n\geq1$. In this case, we write $E_f$ in place of $A_f$. By a result of Carayol (\cite[\S 0.8, Corollaire]{carayol}), $E_f$ has conductor $N$; since $N$ is square-free, $E_f$ is semistable. 

Recall the set $\Xi_f$ of good ordinary primes for $f$ introduced in \S \ref{heegner-subsec} and define
\[ \Theta_f\defeq\Xi_f\cap\bigl\{\text{primes $\ell$ such that $\ell\nmid(q^2-1)$ for any prime $q\,|\,N$}\bigr\}. \]
Pick $p\in\Theta_f$. Let $\Sel_{p^\infty}(E_f/\Q)$ be the $p^\infty$-Selmer group of $E_f$ over $\Q$ and let $\Sha_{p^\infty}(E_f/\Q)$ be the $p$-primary part of the Shafarevich--Tate group of $E_f$ over $\Q$. 

\begin{theorem}[W. Zhang] \label{zhang-thm}
The following properties are equivalent:
\begin{enumerate}
\item $\mathrm{rank}_\Z E_f(\Q)=1$ and $\Sha_{p^\infty}(E_f/\Q)$ is finite;
\vskip 1mm
\item $\ord_{s=1}L(E_f,s)=1$. 
\end{enumerate}
\end{theorem}

\begin{proof} The implication $(2)\Rightarrow(1)$ is the Kolyvagin--Gross--Zagier theorem. Let us assume (1), which ensures that $\corank_{\Z_p}\!\Sel_{p^\infty}(E/\Q)=1$. Since $E_f$ is semistable and $p\geq11$, the representation $\bar\rho_{f,p}$ is surjective by \cite[Theorem 4]{Mazur-isogenies}. Furthermore, since $p\nmid(q^2-1)$ for any prime number $q\,|\,N$, hypothesis (2) in \cite[Theorem 1.4]{zhang-selmer} is trivially satisfied. Therefore, (2) follows from \cite[Theorem 1.4, (ii)]{zhang-selmer}.  \end{proof}

We make a few remarks about Theorem \ref{zhang-thm}.

\begin{remark}
The assumption that $p\nmid(q^2-1)$ for any prime $q\,|\,N$ can be replaced with the weaker (but somewhat less natural) condition that if $q\,|\,N$ is a prime such that $q\equiv\pm1\pmod{p}$, then $\bar\rho_{f,p}$ is ramified at $q$ (see hypothesis (2) in \cite[Theorem 1.4]{zhang-selmer}). 
\end{remark}

\begin{remark}
If one replaces the finiteness of $\Sha_{p^\infty}(E_f/\Q)$ with the finiteness of the full Shafarevich--Tate group $\Sha(E_f/\Q)$ of $E_f$ over $\Q$, then one can take $p\in\Xi_f$ and use \cite[Theorem 1.5]{zhang-selmer} in place of \cite[Theorem 1.4, (ii)]{zhang-selmer}.
\end{remark}

\begin{remark}
Assuming that either
\begin{itemize}
\item[(a)] there is at least one odd prime at which $E_f$ has non-split multiplicative reduction or there are at least two odd primes at which $E_f$ has split multiplicative reduction
\end{itemize}
or
\begin{itemize}
\item[(b)] $E_f$ has split multiplicative reduction at an odd prime (plus some other mild technical conditions),
\end{itemize}
a converse to the Kolyvagin--Gross--Zagier theorem for $E_f$ is provided also by \cite[Theorem A']{Skinner} if (a) holds and by \cite[Theorem A]{Ven} if (b) holds.
\end{remark}

Since $L(E_f,s)=L(f,s)$, by Theorem \ref{zhang-thm} we can reformulate Assumption \ref{main-ass}, which has been in force from \S \ref{choice-subsec}, as

\begin{assumption} \label{main-ass2}
$\rank_\Z E_f(\Q)=1$ and $\#\Sha_{p^\infty}(E_f/\Q)<\infty$.
\end{assumption}

Let $\wp$ be an arithmetic prime of $\mathcal R$ of weight $k_\wp>2$ such that $k_\wp\equiv2\pmod{2(p-1)}$ and trivial character. As before, set $g\defeq f_\wp^\flat\in S_{k_\wp}^\new(\Gamma_0(N))$. Recall that, in this context, $\fP$ is used as a shorthand for $\fP_g=\fP\cap\cO_{\Q_g}$. As in \eqref{L-functional-eq}, write $\varepsilon(\star)$ for the root number of $\star$. 

As in \S \ref{lambda-subsec}, let us consider the quotient $W^\dagger_g=V^\dagger_g\big/T^\dagger_g$. For any number field $L$, let us write $\Sel_\fP(g/L)$ for the $\fP_g$-primary Selmer group of $g$ over $L$ as defined, in terms of $W_g^\dagger$, in \cite[Definition 2.6]{LV} (notice that $\Sel_\fP(g/L)$ is denoted by $H^1_f(L,W_\p)$ in \cite{LV}, where $W_\p$ stands for $W_g^\dagger$). As in \cite{Nek}, we define the $\fP_g$-primary Shafarevich--Tate group of $g$ over $L$ via the short exact sequence of $\cO_{\Q_g,\fP}$-modules
\begin{equation} \label{sha-def-eq}
0\longrightarrow\Lambda_{g,\fP}(L)\otimes_{\Z_p}\Q_p/\Z_p\longrightarrow\Sel_\fP(g/L)\longrightarrow\Sha_\fP(g/L)\longrightarrow0, 
\end{equation}
where the first non-trivial map on the left is essentially a consequence of work of Saito on the weight-monodromy conjecture for compactified Kuga--Sato varieties (\cite{saito}, \cite{saito2}) and of results of Nekov\'a\v{r} (\cite{Nek3}) and Nizio\l\ (\cite{niziol}) on $p$-adic regulators (see, \emph{e.g.}, \cite[\S 2.4]{LV} for details). The group $\Sha_\fP(g/L)$ should be thought of as a higher weight counterpart of the classical Shafarevich--Tate group of an abelian variety, which is given by the recipe ``Selmer group modulo rational points''.

\begin{remark} \label{sha-rem}
Let $K$ be the imaginary quadratic field that was fixed in \S \ref{heegner-subsec}. If $y_{\wp,K}$ is non-torsion, then there is an equality
\begin{equation} \label{sha-def-eq2}
\Sha_\fP(g/K)=\Sel_\fP(g/L)\big/\Sel_\fP(g/K)_{\ddiv}, 
\end{equation}
where $\Sel_\fP(g/K)_{\ddiv}$ is the maximal divisible $\cO_{\Q_g,\fP}$-submodule of $\Sel_\fP(g/K)$. In fact, by the results in \cite{Nek}, if $y_{\wp,K}$ is non-torsion, then both $\Lambda_{g,\fP}(K)\otimes_{\Z_p}\Q_p/\Z_p$ and $\Sel_\fP(g/K)$ have $\cO_{\Q_g,\fP}$-corank $1$, hence $\Lambda_{g,\fP}(K)\otimes_{\Z_p}\Q_p/\Z_p=\Sel_\fP(g/K)_{\ddiv}$. Therefore, equality \eqref{sha-def-eq2} shows that, in this case, the Shafarevich--Tate group \emph{\`a la} Nekov\'a\v{r} defined via \eqref{sha-def-eq} coincides with the Shafarevich--Tate group as introduced by Bloch--Kato in \cite{BK} (\emph{cf.} also \cite{Flach}). Furthermore, as in \S \ref{lambda-subsec}, let $\pi_g$ be a uniformizer for $\cO_{\Q_g,\fP}$. Since $p\cO_{\Q_g,\fP}=\pi_g^e\cO_{\Q_g,\fP}$ for some integer $e\geq1$, the $\cO_{\Q_g,\fP}$-submodule $\Sel_\fP(g/L)_{\ddiv}$ is the maximal $p^\infty$-divisible $\cO_{\Q_g,\fP}$-submodule of $\Sel_\fP(g/L)$.
\end{remark}

Let the imaginary quadratic field $K$ be as in \S \ref{heegner-subsec} and denote by $\Sel_\fP(g/K{)}^\pm$ the $\pm1$-eigenspaces of complex conjugation $\tau\in\Gal(K/\Q)$ acting on $\Sel_\fP(g/K)$, so that there is a splitting
\[ \Sel_\fP(g/K)=\Sel_\fP(g/K{)}^+\oplus\Sel_\fP(g/K{)}^- \]
of $\cO_{\Q_g,\fP}$-modules. 

\begin{lemma} \label{selmer-lemma}
Restriction induces an isomorphism $\Sel_\fP(g/\Q)\simeq\Sel_\fP(g/K{)}^+$ of $\cO_{\Q_g,\fP}$-modules.
\end{lemma}

\begin{proof} Set $W_g^\dagger(K)\defeq H^0\bigl(K,W_g^\dagger\bigr)$. Combining the inflation-restriction exact sequence 
\[ \begin{split}
   0&\longrightarrow H^1\bigl(\Gal(K/\Q),W^\dagger_g(K)\bigr)\longrightarrow H^1\bigl(\Q,W^\dagger_g\bigr)\longrightarrow H^1\bigl(K,W^\dagger_g\bigr)^{\Gal(K/\Q)}\\
     &\longrightarrow H^2\bigl(\Gal(K/\Q),W^\dagger_g(K)\bigr)\longrightarrow\dots
   \end{split} \]
with the triviality of $W^\dagger_g(K)$ that is guaranteed by Corollary \ref{solvable-coro}, one gets an isomorphism $H^1\bigl(\Q,W^\dagger_g\bigr)\simeq H^1\bigl(K,W^\dagger_g\bigr)^{\Gal(K/\Q)}$. By keeping track of local conditions, it can be checked that this isomorphism restricts to the desired isomorphism $\Sel_\fP(g/\Q)\simeq\Sel_\fP(g/K{)}^+$. \end{proof}

\begin{proposition} \label{sha-prop}
If $\Sha_\fP(g/K)$ is finite, then $\Sha_\fP(g/\Q)$ is finite. 
\end{proposition}

\begin{proof} Let us set $\Lambda_\Q\defeq\Lambda_{g,\fP}(\Q)$, $\Lambda_K\defeq\Lambda_{g,\fP}(K)$, $S_\Q\defeq\Sel_\fP(g/\Q)$, $S_K\defeq\Sel_\fP(g/K)$, $\Sha_\Q\defeq\Sha_\fP(g/\Q)$, $\Sha_K\defeq\Sha_\fP(g/K)$, $\cG\defeq\Gal(K/\Q)$. With notation as above, the short exact sequence of $\cO_{\Q_g,\fP}$-modules
\[ 0\longrightarrow\Lambda_K\otimes_{\Z_p}\Q_p/\Z_p\longrightarrow S_K\longrightarrow\Sha_K\longrightarrow0 \]
yields a long exact sequence
\[ 0\longrightarrow(\Lambda_K\otimes_{\Z_p}\Q_p/\Z_p)^+\longrightarrow S_K^+\longrightarrow\Sha_K^+\longrightarrow H^1(\cG,\Lambda_K\otimes_{\Z_p}\Q_p/\Z_p)\longrightarrow\dots \]
in cohomology. Since $\#\cG=2$ and $p$ is odd, $H^1(\cG,\Lambda_K\otimes_{\Z_p}\Q_p/\Z_p)=0$; therefore, we get an isomorphism 
\begin{equation} \label{sha-eq0}
\Sha_K^+\simeq S_K^+\big/(\Lambda_K\otimes_{\Z_p}\Q_p/\Z_p)^+ 
\end{equation}
of $\cO_{g,\fP}$-modules. By Corollary \ref{no-torsion-coro}, $\Lambda_K$ is torsion-free over $\Z_p$, so there is an identification 
\begin{equation} \label{sha-eq1}
(\Lambda_K\otimes_{\Z_p}\Q_p)^+=\Lambda_K^+\otimes_{\Z_p}\Q_p 
\end{equation}
of $\cO_{\Q_g,\fP}$-modules. Since $\Lambda_K$ is flat over $\cO_{\Q_g,\fP}$, there is a short exact sequence of $\cO_{\Q_g,\fP}$-modules
\[ 0\longrightarrow\Lambda_K\longrightarrow\Lambda_K\otimes_{\Z_p}\Q_p\longrightarrow\Lambda_K\otimes_{\Z_p}\Q_p/\Z_p\longrightarrow0. \]
Reasoning as above, we obtain an isomorphism 
\begin{equation} \label{sha-eq2}
(\Lambda_K\otimes_{\Z_p}\Q_p/\Z_p)^+\simeq(\Lambda_K\otimes_{\Z_p}\Q_p)^+\big/\Lambda_K^+=(\Lambda_K^+\otimes_{\Z_p}\Q_p)\big/\Lambda_K^+ 
\end{equation}
of $\cO_{g,\fP}$-modules, where the equality on the right follows from \eqref{sha-eq1}. Similarly, there is an isomorphism of $\cO_{g,\fP}$-modules
\begin{equation} \label{sha-eq3}
\Lambda_\Q\otimes_{\Z_p}\Q_p/\Z_p\simeq(\Lambda_\Q\otimes_{\Z_p}\Q_p)\big/\Lambda_\Q. 
\end{equation}
Moreover, by Proposition \ref{lambda-prop2}, restriction gives an isomorphism of $\cO_{g,\fP}$-modules $\Lambda_\Q\simeq\Lambda_K^+$. Thus, from \eqref{sha-eq2} and \eqref{sha-eq3} we get an isomorphism of $\cO_{g,\fP}$-modules
\begin{equation} \label{sha-eq4}
(\Lambda_K\otimes_{\Z_p}\Q_p/\Z_p)^+\simeq\Lambda_\Q\otimes_{\Z_p}\Q_p/\Z_p.
\end{equation}
On the other hand, by Lemma \ref{selmer-lemma}, restriction induces also an isomorphism $S_\Q\simeq S_K^+$ of $\cO_{\Q_g,\fP}$-modules. Finally, combining \eqref{sha-eq0} and \eqref{sha-eq4} produces isomorphisms 
\[ \Sha_K^+\simeq S_K^+\big/(\Lambda_K\otimes_{\Z_p}\Q_p/\Z_p)^+\simeq S_\Q\big/(\Lambda_\Q\otimes_{\Z_p}\Q_p/\Z_p)=\Sha_\Q \]
of $\cO_{g,\fP}$-modules. In particular, $\Sha_\Q$ is finite if $\Sha_K$ is finite. \end{proof}

The next result says that the finiteness of the $p$-primary Shafarevich--Tate group of $f$ together with the algebraic rank $1$ property (Assumption \ref{main-ass2}) propagates to certain higher (even) weight forms in our Hida family $\hf$.

\begin{theorem} \label{sha-thm}
$r_{\alg,\fP}(g/\Q)=1$ and $\#\Sha_\fP(g/\Q)<\infty$.
\end{theorem}

\begin{proof} Let us first prove the rank part. Let $K$ be the imaginary quadratic field that was chosen in \S \ref{heegner-subsec}. Let $\tilde\tau\in\Gal(K_1/\Q)$ be the restriction to $K_1$ of complex conjugation in $G_\Q$, so that $\tilde\tau$ is a lift of $\tau$. By \cite[Proposition 6.2]{Nek}, there is an equality
\begin{equation} \label{tau-eq}
\tilde\tau(y_{\wp,1})=-\varepsilon(g)\cdot\sigma(y_{\wp,1})
\end{equation}
for a suitable $\sigma\in\Gal(K_1/K)$. The arithmetic prime $\wp_2$ of $\mathcal R$ corresponding to $f$ is generic (see, \emph{e.g.}, \cite[(3.1.1)]{NP}), hence $\varepsilon(\hf)=\varepsilon(f)=-1$. Lemmas \ref{L-functions-lemma} and \ref{root-number-lemma} imply that $\varepsilon(g)=\varepsilon(\hf)=-1$, therefore \eqref{tau-eq} gives
\[ \tilde\tau(y_{\wp,1})=\sigma(y_{\wp,1}) \]
for some $\sigma\in\Gal(K_1/K)$. By taking corestrictions and using the fact that $\tilde\tau\delta=\delta\tilde\tau$ for all $\delta\in\Gal(K_1/K)$, we get
\begin{equation} \label{tau-cores-eq}
\begin{split}
   \tilde\tau\bigl(\cores_{K_1/K}(y_{\wp,1})\bigr)&=\cores_{K_1/K}\bigl(\tilde\tau(y_{\wp,1})\bigr)\\
   &=\cores_{K_1/K}\bigl(\sigma(y_{\wp,1})\bigr)=\cores_{K_1/K}(y_{\wp,1}). 
\end{split} 
\end{equation}
Since $\cores_{K_1/K}(y_{\wp,1})=y_{\wp,K}$, we deduce that $\tau(y_{\wp,K})=y_{\wp,K}$, \emph{i.e.}, $y_{\wp,K}\in\Lambda_{g,\fP}(\Q)\otimes_\Z\Q$ by Proposition \ref{lambda-prop2} (see also Remark \ref{lambda-inj-rem}). But Theorem \ref{non-torsion-prop} ensures that $y_{\wp,K}\not=0$, hence $r_{\alg,\fP}(g/\Q)=\dim_{\Q_{g,\fP}}\bigl(\Lambda_{g,\fP}(\Q)\otimes_\Z\Q\bigr)\geq1$. On the other hand, since $y_{\wp,K}$ is non-torsion, $r_{\alg,\fP}(g/K)=1$ by part (1) of Theorem \ref{nekovar-thm}, whence $r_{\alg,\fP}(g/\Q)\leq1$ by Proposition \ref{lambda-prop1}.

Finally, $\Sha_\fP(g/K)$ is finite by part (2) of Theorem \ref{nekovar-thm} and Remark \ref{sha-rem}, and then the finiteness of $\Sha_\fP(g/\Q)$ follows from Proposition \ref{sha-prop}. \end{proof}

The only condition that we imposed on $p$ is that it belong to $\Theta_f$, whose complement in the set of prime numbers that are ordinary for $f$ is finite (\emph{cf.} Remark \ref{sigma-rem}), so Theorem \ref{sha-thm} implies the rank $1$ part of Theorem A.

\begin{remark}
At the cost of adding some extra hypotheses on the local components of the cuspidal automorphic representation of $\GL_2(\A_\Q)$ attached to $f$, where $\A_\Q$ is the adele ring of $\Q$, we could avoid assuming that $A_f$ is an elliptic curve. More precisely, we could replace Assumption \ref{main-ass2} with the conditions $\rank_\Z A_f(\Q)=[\Q_f:\Q]$ and $\#\Sha(A_f/\Q)<\infty$, using \cite[Theorem A]{Skinner} instead of Theorem \ref{zhang-thm}.
\end{remark}

\begin{remark}
It might be tempting to relax Assumption \ref{main-ass2} by only requiring $\Sha_{p^\infty}(E_f/\Q)$ to be finite and then hopefully deduce the finiteness of $\Sha_\fP(g/\Q)$. Unfortunately, such a stronger result appears to lie well beyond the scope of current techniques.
\end{remark}

\begin{corollary} \label{sha-coro}
$r_{\alg,\fP}(g/\Q)=r_{\min}\bigl(\hf\bigr)$.
\end{corollary}

\begin{proof} Since $\varepsilon(\hf)=-1$, we have $r_{\min}(\hf)=1$, and the corollary follows immediately from Theorem \ref{sha-thm}. \end{proof}

As was remarked in the introduction, this result is consistent with (and provides some evidence for) Conjecture \ref{main-conj}.

\section{Analytic rank one in Hida families} \label{rank-sec}

As in Section \ref{sha-sec}, choose any $p\in\Xi_f$ and consider the Hida family $\hf\in\mathcal R[\![q]\!]$. Recall that $\varepsilon(\hf)=\varepsilon(f)=-1$, hence
\begin{equation} \label{r-min-eq}
r_{\min}\bigl(\hf\bigr)=1.
\end{equation}
In this section, we prove our result (Theorem \ref{main-thm}) in the direction of Greenberg's conjecture on analytic ranks in Hida families (Conjecture \ref{greenberg-conj}). Contrary to what was done in \S \ref{sha-subsec}, we do not assume that the Fourier coefficients of $f$ are rational.

\subsection{Zhang's formula for derivatives of $L$-functions} \label{zhang-subsec}

Let $g\in S_k(\Gamma_0(N))$ be a newform of weight $k\geq4$ and level $N$. Here $N$ is the level of our weight $2$ form $f$ that was introduced in \S \ref{choice-subsec}, so that everything we will say below applies, in particular, to the newforms $f_\wp^\flat$ in our $p$-adic Hida family $\f^{(p)}$.

Using arithmetic intersection theory \emph{\`a la} Gillet--Soul\'e (\cite{GS-1}, \cite[Chapter III]{soule}), S.-W. Zhang defined in \cite{Zhang-heights} a pairing, which we denote by ${\langle\cdot,\cdot\rangle}_{g,\GS}$, between certain CM cycles on Kuga--Sato varieties. Using this pairing, he proved a higher weight analogue of the Gross--Zagier formula for the $L$-function of a weight $2$ newform. More precisely, let $K$ be the imaginary quadratic field from \S \ref{heegner-subsec} and write $h_K$ for its class number. Furthermore, denote by $(g,g)$ the Petersson inner product of $g$ with itself and set $u_K\defeq\#\cO_K^\times/2$. Zhang's formula expresses the critical value of the derivative of $L(g/K,s)$ in terms of a suitable combination $s_g'$ of Heegner cycles in $\R$-linear Chow groups of Kuga--Sato varieties fibered over (classical) modular curves.

\begin{theorem}[S.-W. Zhang] \label{zhang-thm2}
$L'(g/K,k/2)=\begin{displaystyle}\frac{2^{2k-1}\pi^k(g,g)}{(k-2)!u_K^2h_K\sqrt{|D_K|}}\end{displaystyle}{\big\langle s_g',s_g'\big\rangle}_{g,\GS}$.
\end{theorem}

\begin{proof} With notation as in \cite{Zhang-heights}, this follows from \cite[Corollary 0.3.2]{Zhang-heights} upon taking $\chi$ to be the trivial character. \end{proof}

\begin{remark}
Theorem \ref{zhang-thm2} is valid, more generally, for newforms of any level $M$ and for any imaginary quadratic field satisfying the Heegner hypothesis relative to $M$.  
\end{remark}

Some words of caution are in order here. Unlike what happens with the N\'eron--Tate height pairing on abelian varieties that appears in the Gross--Zagier formula, the non-degeneracy of the Gillet--Soul\'e pairing is only conjectural. In fact, the non-degeneracy of ${\langle\cdot,\cdot\rangle}_{g,\GS}$ is (a consequence of) one of the arithmetic analogues of the standard conjectures proposed by Gillet and Soul\'e (\cite[Conjecture 2]{GS-2}); it is also a special case of general conjectures of Beilinson (\cite{beilinson-height}) and Bloch (\cite{bloch-height}) on positive definiteness of height pairings.

\subsection{Results over $K$} \label{quadratic-subsec}

We are interested in the analytic rank of forms in our Hida family $\hf$ after base change to the imaginary quadratic field $K$ that was fixed in \S \ref{heegner-subsec}. As before, the Dirichlet character associated with $K$ will be denoted by $\chi_K$. 

For any arithmetic prime $\wp$ of $\mathcal R$ with trivial character, set $s_\wp'\defeq s'_{f_\wp^\flat}$ and ${\langle\cdot,\cdot\rangle}_{\wp,\GS}\defeq{\langle\cdot,\cdot\rangle}_{f^\flat_\wp,\GS}$. From now until the end of \S \ref{Q-subsec} we require the following non-degeneracy condition to hold.

\begin{assumption} \label{main-ass3}
The pairing ${\langle\cdot,\cdot\rangle}_{\wp,\GS}$ is non-degenerate for every arithmetic prime $\wp$ of weight $k_\wp\equiv2\pmod{2(p-1)}$ and trivial character.
\end{assumption}

Of course, the validity of Assumption \ref{main-ass3} would be guaranteed by the non-degeneracy of ${\langle\cdot,\cdot\rangle}_{g,\GS}$ for all $g$ that is predicted in \cite{GS-2}.

\begin{remark}
Unfortunately, while it is natural to impose a non-degeneracy condition like Assumption \ref{main-ass3} when studying the arithmetic of Heegner cycles (see, \emph{e.g.}, \cite[Assumption 4.1]{Xue}), we are not aware of any result in this direction for modular forms of weight $>2$.
\end{remark}

The proof of the following theorem, which may look deceptively straightforward, makes crucial use of the non-torsionness result of Theorem \ref{non-torsion-prop}.

\begin{theorem} \label{quadratic-thm}
If $\wp$ is an arithmetic prime of $\mathcal R$ with trivial character and weight $k_\wp>2$ such that $k_\wp\equiv2\pmod{2(p-1)}$, then $r_\an(f_\wp^\flat/K)=1$.
\end{theorem}

\begin{proof} Let $\wp$ be an arithmetic prime of $\mathcal R$ as in the statement of the theorem. Proposition \ref{analytic-rank-K-prop} with $g=f_\wp^\flat$ gives 
\begin{equation} \label{order-K-eq}
r_\an(f_\wp^\flat/K)\geq1.
\end{equation}   
If $L'(f_\wp^\flat/K,k_\wp/2)=0$, then Theorem \ref{zhang-thm2} and Assumption \ref{main-ass3} give $s'_\wp=0$. By unfolding the definition of $s'_\wp$, it can be shown that the vanishing of $s_\wp'$ forces $y_{\wp,K}$ to be torsion (\cite[Proposition 4.10, (2)]{LV-tamagawa}), which contradicts Theorem \ref{non-torsion-prop}. Therefore $L'(f_\wp^\flat/K,k_\wp/2)\not=0$, and then \eqref{order-K-eq} yields $r_\an(f_\wp^\flat/K)=1$. \end{proof}

\subsection{Results over $\Q$} \label{Q-subsec}

The next theorem, which is proved under Assumption \ref{main-ass3}, is our main result towards Conjecture \ref{greenberg-conj}. 

\begin{theorem} \label{main-thm}
If $\wp$ is an arithmetic prime of $\mathcal R$ with trivial character and weight $k_\wp>2$ such that $k_\wp\equiv2\pmod{2(p-1)}$, then $r_\an(f_\wp^\flat)=1=r_{\min}(\hf)$.
\end{theorem}

\begin{proof} We noticed in \eqref{r-min-eq} that $r_{\min}(\hf)=1$. Let $\wp$ be an arithmetic prime of $\mathcal R$ as in the statement of the theorem. Lemma \ref{root-number-lemma} gives $\varepsilon(f^\flat_\wp)=\varepsilon(\hf)=-1$, so  
\begin{equation} \label{order-eq}
r_\an(f_\wp^\flat)\geq1.
\end{equation}
Let $K$ be the imaginary quadratic field, with associated Dirichlet character $\chi_K$, from \S \ref{heegner-subsec}. Theorem \ref{quadratic-thm} says that
\begin{equation} \label{order-f-k-K-eq}
r_\an(f_\wp^\flat/K)=1. 
\end{equation}
On the other hand, from \eqref{base-change-L-eq} with $g=f_\wp^\flat$ we obtain
\begin{equation} \label{ranks-sum-eq}
r_\an(f_\wp^\flat/K)=r_\an(f_\wp^\flat)+r_\an(f_\wp^\flat\otimes\chi_K).
\end{equation}
The theorem follows by combining \eqref{order-eq}, \eqref{order-f-k-K-eq} and \eqref{ranks-sum-eq}. \end{proof}

\begin{corollary} \label{main-coro}
If $\wp$ is an arithmetic prime of $\mathcal R$ with trivial character and weight $k_\wp>2$ such that $k_\wp\equiv2\pmod{2(p-1)}$, then $r_\an(f_\wp)=1=r_{\min}(\hf)$.
\end{corollary}

\begin{proof} Immediate from Lemma \ref{L-functions-lemma} and Theorem \ref{main-thm}. \end{proof}

Since $p$ is allowed to be any element of $\Xi_f$, which rules out only finitely many primes that are ordinary for $f$, Theorem \ref{main-thm} implies Theorem B and offers, to the best of our knowledge, the first evidence in higher weight for Conjecture \ref{greenberg-conj} when the minimal admissible generic rank of the Hida family is $1$ (or, equivalently, the root number of the Hida family is $-1$).

\begin{remark}
By \cite[Theorem 8]{Howard-derivatives}, a result analogous to Theorem \ref{main-thm} holds also for all but finitely many arithmetic primes of weight $2$.
\end{remark}

\begin{remark}
For a result towards a $p$-adic variant of Conjecture \ref{greenberg-conj} (\emph{cf.} Remark \ref{p-adic-variant-rem}), see \cite[Theorem 5.9]{CW}. Notice that the role of Assumption \ref{main-ass3} is played in \cite{CW} by the assumption that a certain cyclotomic $\mathcal R$-adic height pairing is non-degenerate.
\end{remark}

\section{Shafarevich--Tate groups in Hida families: the rank zero case} \label{sha-sec2}

The aim of this final section is to prove our results (Theorem \ref{sha-thm2}) on Shafarevich--Tate groups and on algebraic $\fP$-ranks for a large class of higher (even) weight modular forms in a Hida family $\hf$ in a rank $0$ setting. We remark that one could also obtain results of this kind using as a key ingredient the Mazur--Kitagawa two-variable $p$-adic $L$-function, which plays no role in our arguments (\emph{cf.} Remark \ref{mk-rem}).

\subsection{The newform $f$ of weight $2$ and choice of $p$} \label{choice-subsec2}

As before, let $f\in S_2(\Gamma_0(N))$ be a normalized newform of weight $2$ and square-free level $N$. In this last section, unlike what was done in \S \ref{choice-subsec}, we make the following

\begin{assumption} \label{main-0-ass}
$r_\an(f)=0$.
\end{assumption}

As a consequence, $\varepsilon(f)=1$. By \cite[p. 543, Theorem, (i)]{BFH} (see also \cite{MM-derivatives}), there is an imaginary quadratic field $K'$, with associated Dirichlet character $\chi_{K'}$, such that
\begin{itemize}
\item[(a)] the primes dividing $N$ split in $K'$;
\item[(b)] $r_\an(f\otimes\chi_{K'})=1$.
\end{itemize}
Fix such a field $K'$. As in \S \ref{heegner-subsec}, the theory of complex multiplication allows one to introduce a Heegner point $\alpha_{K'}\in A_f(K')$. We deduce from \eqref{base-change-L-eq} and condition (b) that $r_\an(f/K')=1$, which, by the Gross--Zagier formula, is equivalent to $\alpha_{K'}$ being non-torsion. Using $K'$ in place of $K$, we define a set $\Omega_f$ of prime numbers exactly as the set $\Xi_f$ from \S \ref{heegner-subsec}. In particular, $\Omega_f$ has density $1$ (\emph{cf.} Proposition \ref{density-prop}) and consists of all but finitely many primes that are ordinary for $f$ in the sense of \S \ref{choice-subsec} (\emph{cf.} Remark \ref{sigma-rem}). 

Pick $p\in\Omega_f$ and let $\hf\in\mathcal R'[\![q]\!]$ be the $p$-adic Hida family through $f$, where $\mathcal R'$ is a complete local noetherian domain that is finite and flat over the Iwasawa algebra $\cO_{\Q_f,\p}[\![\Gamma]\!]$. 

\subsection{Results on Shafarevich--Tate groups} \label{sha-subsec2}

Let $\wp$ be an arithmetic prime of $\mathcal R'$ of weight $k_\wp>2$ such that $k_\wp\equiv2\pmod{2(p-1)}$ and trivial character. As briefly reviewed in \S \ref{cycles-subsec}, there exists a systematic supply of Heegner cycles $y_{\wp,c}\in H^1\Big(K'_c,T^\dagger_{f_\wp^\flat}\Big)$, where $K'_c$ is the ring class field of $K'$ of conductor $c$. Let us define the $K'$-rational Heegner cycle
\[ y_{\wp,K'}\defeq\cores_{K'_1/K'}(y_{\wp,1})\in H^1\Big(K',T^\dagger_{f_\wp^\flat}\Big),  \]
which is the counterpart over $K'$ of the cycle $y_{\wp.K}$ from \eqref{y-K-eq}. In particular, one can prove the analogue of Theorem \ref{nekovar-thm} with $K'$ in place of $K$.

\begin{theorem} \label{non-torsion-prop2}
The Heegner cycle $y_{\wp,K'}$ is non-torsion over $\cO_{f_\wp^\flat,\fP}$.
\end{theorem}

\begin{proof} Proceed exactly as in the proof of Theorem \ref{non-torsion-prop}. \end{proof}

As in \S \ref{sha-subsec}, suppose that $A_f$ is an elliptic curve and denote it by $E_f$. In particular, since $N$ is square-free, $E_f$ is semistable. 

As a consequence of results of Gross--Zagier and Kolyvagin, Assumption \ref{main-0-ass} implies that $\rank_\Z E_f(\Q)=0$ (\emph{i.e.}, $E_f(\Q)$ is finite) and $\Sha_{p^\infty}(E_f/\Q)$ is finite (the semistability of $E_f$ plays no role here). In fact, the converse is also true: by results of Skinner and Urban on the Iwasawa Main Conjecture (\cite[Theorem 2, (b)]{SU}), the semistability of $E_f$ and the fact that $p\geq11$ ensure that Assumption \ref{main-0-ass} is equivalent to

\begin{assumption} \label{main-0-ass2}
$\rank_\Z E_f(\Q)=0$ and $\#\Sha_{p^\infty}(E_f/\Q)<\infty$.
\end{assumption} 

Let $\wp$ be an arithmetic prime of $\mathcal R'$ of weight $k_\wp>2$ such that $k_\wp\equiv2\pmod{2(p-1)}$ and trivial character. Set $g\defeq f_\wp^\flat$. The next result is the rank $0$ counterpart of Theorem \ref{sha-thm}.

\begin{theorem} \label{sha-thm2}
$r_{\alg,\fP}(g/\Q)=0$ and $\#\Sha_\fP(g/\Q)<\infty$.
\end{theorem}

\begin{proof}  Let us prove the rank part first.  Denote by $\tilde\tau\in\Gal(K'_1/\Q)$ the restriction to $K'_1$ of complex conjugation in $G_\Q$, so that $\tilde\tau$ is a lift of a generator $\tau$ of $\Gal(K'/\Q)$. By \cite[Proposition 6.2]{Nek}, there is an equality
\begin{equation} \label{tau-eq2}
\tilde\tau(y_{\wp,1})=-\varepsilon(g)\cdot\sigma(y_{\wp,1})
\end{equation}
for a suitable $\sigma\in\Gal(K'_1/K')$. On the other hand, $\varepsilon(g)=\varepsilon(\hf)=\varepsilon(f)=1$, hence \eqref{tau-eq2} gives
\begin{equation} \label{tau-sigma-eq}
\tilde\tau(y_{\wp,1})=-\sigma(y_{\wp,1}) 
\end{equation}
for some $\sigma\in\Gal(K'_1/K')$. Using \eqref{tau-sigma-eq} and proceeding as in \eqref{tau-cores-eq}, we see that
\begin{equation} \label{tau-K'-eq}
\tau(y_{\wp,K'})=-y_{\wp,K'}.
\end{equation}
Arguing by contradiction, now suppose that there exists $y_\wp\in\Lambda_{g,\fP}(\Q)$ that is non-torsion over $\cO_{\Q_g,\fP}$: we show that $y_\wp$ and $y_{\wp,K'}$ are linearly independent over $\cO_{\Q_g,\fP}$ (here recall that $\Lambda_{g,\fP}(\Q)\subset\Lambda_{g,\fP}(K)$ by Proposition \ref{lambda-prop1}). If this is not the case, then there are $a,b\in\cO_{\Q_g,\fP}$ such that $a\not=0\not=b$ and
\[ ay_{\wp,K'}=by_\wp\in\Lambda_{g,\fP}(\Q). \]
It follows from \eqref{tau-K'-eq} that
\[ ay_{\wp,K'}=\tau(ay_{\wp,K'})=-ay_{\wp,K'}, \]
which shows that $2ay_{\wp,K'}=0$. This contradicts the fact that, by Theorem \ref{non-torsion-prop2}, $y_{\wp,K'}$ is non-torsion over $\cO_{\Q_g,\fP}$. Therefore the $\cO_{\Q_g,\fP}$-module generated by $y_\wp$ and $y_{\wp,K'}$ has rank $2$, which is impossible because $r_{\alg,\fP}(g/K')=1$ by part (1) of Theorem \ref{nekovar-thm}. We conclude that $r_{\alg,\fP}(g/\Q)=0$, as claimed. 

Finally, since $y_{\wp,K'}$ is non-torsion, $\Sha_\fP(g/K')$ is finite by part (2) of Theorem \ref{nekovar-thm} and Remark \ref{sha-rem}, and then $\Sha_\fP(g/\Q)$ is finite by Proposition \ref{sha-prop}. \end{proof}

Similarly to the rank $1$ case, the only condition that we imposed on $p$ is that it belong to $\Omega_f$, whose complement in the set of primes that are ordinary for $f$ is finite. Therefore, Theorem \ref{sha-thm2} implies the rank $0$ part of Theorem A.

\begin{corollary} \label{sha-coro2}
$r_{\alg,\fP}(g/\Q)=r_{\min}\bigl(\hf\bigr)$.
\end{corollary}

\begin{proof} Since $\varepsilon(\hf)=1$, we have $r_{\min}(\hf)=0$, and Theorem \ref{sha-thm2} immediately implies the corollary. \end{proof}

As remarked in the introduction, this result is consistent with (and offers some evidence for) Conjecture \ref{main-conj}.

\begin{remark} \label{mk-rem}
As we pointed out earlier, results in the vein of Theorem \ref{sha-thm2} could also be obtained using the Mazur--Kitagawa two-variable $p$-adic $L$-function (\cite{kitagawa}; \emph{cf.} also \cite[\S 3.4]{EPW}) as a crucial ingredient. Roughly speaking, in this case the arguments would go as follows. By \cite[Theorem 7]{Howard-derivatives}, whose proof exploits properties of the above-mentioned $p$-adic $L$-function, Assumption \ref{main-0-ass2} (which is equivalent to Assumption \ref{main-0-ass}) implies that $r_\an(f_\wp^\flat)=0$ for all but finitely many arithmetic primes $\wp$ of weight larger than $2$ and trivial character. In order to deduce from this analytic fact that $r_{\alg,\fP}(f_\wp^\flat/\Q)=0$ and $\#\Sha_\fP(f_\wp^\flat/\Q)<\infty$ for all such $\wp$, one would then use Kolyvagin-type arguments similar to those described in \S \ref{choice-subsec2} and \S \ref{sha-subsec2}. In light of what we have just said, from our point of view our strategy to prove Theorem \ref{sha-thm2} is especially interesting because it bypasses any consideration whatsoever involving $p$-adic $L$-functions.
\end{remark}

\appendix

\section{Dihedral residual representations} \label{appendix}

As we mentioned in Remark \ref{c-o-rem}, Castella proved Theorem \ref{c-o-thm} in \cite{CasHeeg} under the extra assumption that the restriction ${\bar\rho_{f,\p}|}_{G_K}$ is irreducible. This is a potentially delicate issue, as we want to have the freedom to choose the imaginary quadratic field $K$ according to the prescription in \S \ref{heegner-subsec}. Although we ultimately need not bother about this irreducibility condition, as Ota removed it in \cite{Ota-JNT}, in this short appendix we elaborate on it and show that it is guaranteed by a natural group-theoretic property, namely, that $\bar\rho_{f,\p}$ be not dihedral. This is a well-known result in representation theory, but for lack of a convenient reference we decided to include some details. 

\begin{proposition} \label{induced-prop}
Let $G$ be a group, let $\rho:G\rightarrow\GL_2(L)$ be a degree $2$ representation of $G$ over a field $L$ and let $H$ be a normal subgroup of $G$ of index $2$. If $\rho$ is irreducible and ${\rho|}_H$ is reducible, then $\rho\simeq\Ind^G_H(\varphi)$ for a suitable degree $1$ representation $\varphi$ of $H$ over $L$.
\end{proposition}

\begin{proof} By virtue of the semisimplicity of restriction (see \cite[Exercise 2.3.4]{kowalski}), ${\rho|}_H$ is semisimple as a representation of $H$. Since ${\rho|}_H$ is reducible, it follows that
\begin{equation} \label{res-rep-eq}
{\rho|}_H=\varphi\oplus\psi 
\end{equation}
for suitable degree $1$ subrepresentations $\varphi$ and $\psi$ of $H$. By the adjointness of induction and restriction (an incarnation of Frobenius reciprocity, see \cite[Proposition 2.3.9]{kowalski}), there is a natural isomorphism of $L$-vector spaces
\begin{equation} \label{frobenius-eq}
\Hom_G\bigl(\rho,\Ind^G_H(\varphi)\bigr)\simeq\Hom_H\bigl({\rho|}_H,\varphi\bigr), 
\end{equation}
where $\Hom_\star(\cdot,\cdot)$ with $\star\in\{G,H\}$ denotes the $L$-vector space of morphisms between two representations of $\star$. The splitting in \eqref{res-rep-eq} shows that the right-hand side of \eqref{frobenius-eq} is not trivial, hence $\Hom_G\bigl(\rho,\Ind^G_H(\varphi)\bigr)\neq0$. Since $\rho$ is irreducible, this means that there is an injective morphism of representations of $G$ from $\rho$ to $\Ind^G_H(\varphi)$. On the other hand, by \cite[Proposition 2.3.11]{kowalski} one has
\[ \dim\bigl(\Ind^G_H(\varphi)\bigr)=[G:H]\dim(\varphi)=2=\dim(\rho), \] 
which implies that $\rho$ and $\Ind^G_H(\varphi)$ are isomorphic. \end{proof}

Let $\F$ be a finite field and let $\bar\F$ be an algebraic closure of $\F$.

\begin{definition} \label{dihedral-def}
A (continuous) irreducible representation $\rho:G_\Q\rightarrow\GL_2(\bar\F)$ is \emph{dihedral} if there is a quadratic field $K\subset\bar\Q$ such that $\rho$ is induced from a character of $G_K$.  
\end{definition}

Equivalently, a representation as in Definition \ref{dihedral-def} is dihedral if its projective image is isomorphic to the dihedral group $D_n$ for some $n\geq3$. 

\begin{corollary} \label{appendix-coro}
Let $\rho:G_\Q\rightarrow\GL_2(\bar\F)$ be an irreducible representation. If $\rho$ is not dihedral, then ${\rho|}_{G_K}$ is irreducible for every quadratic field $K$.
\end{corollary}

\begin{proof} Immediate from Proposition \ref{induced-prop}. \end{proof}

In light of Corollary \ref{appendix-coro}, if we chose $f$ so that $\bar\rho_{f,\p}$ is not dihedral, then in \S \ref{big-subsec} we could directly use \cite[Theorem 6.5]{CasHeeg}, up to the minor corrections to \cite{CasHeeg} that are pointed out in \cite[Remark 1.3]{Ota-JNT}.

\bibliographystyle{amsplain}
\bibliography{Rank}

\end{document}